\documentclass[10pt,leqno]{amsart}

\usepackage{amsmath}
\usepackage{amssymb,xfrac}
\usepackage{amsthm}
\usepackage{aliascnt}
\usepackage{ifthen} 
\usepackage[breaklinks,unicode]{hyperref} 
\usepackage[capitalise]{cleveref}
\usepackage[shortlabels]{enumitem} 
\setlist{nosep}
\setlist[enumerate]{label=(\arabic*),ref=(\arabic*)}

\usepackage[dvipsnames]{xcolor}
\usepackage{tikz-cd, tikz}
\usepackage{caption}
\usepackage{subcaption}
\usepackage{mathrsfs} 
\usepackage{float}

\usepackage{cancel}

\numberwithin{equation}{section}

\newtheorem{theorem}{Theorem}[section]
\newaliascnt{claim}{theorem}

\aliascntresetthe{claim}
\newaliascnt{proposition}{theorem}
\newtheorem{proposition}[proposition]{Proposition}
\aliascntresetthe{proposition}
\newaliascnt{lemma}{theorem}
\newtheorem{lemma}[lemma]{Lemma}
\aliascntresetthe{lemma}
\newaliascnt{corollary}{theorem}
\newtheorem{corollary}[corollary]{Corollary}
\aliascntresetthe{corollary}
\newaliascnt{conjecture}{theorem}
\newtheorem{conjecture}[conjecture]{Conjecture}
\aliascntresetthe{conjecture}
\newaliascnt{convention}{theorem}

\aliascntresetthe{convention}

\newtheorem*{theorem*}{Theorem}
\newtheorem*{claim*}{Claim}
\newtheorem*{proposition*}{Proposition}
\newtheorem*{lemma*}{Lemma}
\newtheorem*{corollary*}{Corollary}

\theoremstyle{definition}
\newaliascnt{definition}{theorem}
\newtheorem{definition}[definition]{Definition}
\aliascntresetthe{definition}
\newaliascnt{observation}{theorem}

\aliascntresetthe{observation}
\newaliascnt{remark}{theorem}
\newtheorem{remark}[remark]{Remark}
\aliascntresetthe{remark}
\newaliascnt{example}{theorem}
\newtheorem{example}[example]{Example}
\aliascntresetthe{example}
\newaliascnt{question}{theorem}
\newtheorem{question}[question]{Question}
\aliascntresetthe{question}
\newaliascnt{exercise}{theorem}

\aliascntresetthe{exercise}
\newaliascnt{fact}{theorem}

\aliascntresetthe{fact}
\newaliascnt{notation}{theorem}

\aliascntresetthe{notation}

\crefname{claim}{Claim}{Claims}
\crefname{proposition}{Proposition}{Propositions}
\crefname{lemma}{Lemma}{Lemmas}
\crefname{corollary}{Corollary}{Corollaries}
\crefname{conjecture}{Conjecture}{Conjectures}
\crefname{definition}{Definition}{Definitions}
\crefname{observation}{Observation}{Observations}
\crefname{remark}{Remark}{Remarks}
\crefname{example}{Example}{Examples}
\crefname{question}{Question}{Questions}
\crefname{exercise}{Exercise}{Exercises}
\crefname{fact}{Fact}{Facts}
\crefname{notation}{Notation}{Notations}
\crefname{section}{\S}{\S\S}
\crefname{subsection}{\S}{\S\S}
\crefname{subsubsection}{\S}{\S\S}
\Crefname{section}{\S}{\S\S}
\Crefname{subsection}{\S}{\S\S}
\Crefname{subsubsection}{\S}{\S\S}

\newtheorem*{definition*}{Definition}
\newtheorem*{observation*}{Observation}
\newtheorem*{remark*}{Remark}
\newtheorem*{example*}{Example}
\newtheorem*{question*}{Question}
\newtheorem*{exercise*}{Exercise}
\newtheorem*{fact*}{Fact}
\newtheorem*{notation*}{Notation}

\newcommand{\bbN}{\mathbb{N}}

\newcommand{\bbR}{\mathbb{R}}

\newcommand{\bbZ}{\mathbb{Z}}

\newcommand{\bfG}{\mathbf{G}}
\newcommand{\bfH}{\mathbf{H}}

\newcommand*{\sgen}[1]{\langle#1\rangle}

\DeclareMathOperator{\stab}{stab}

\DeclareMathOperator{\SL}{SL}

\DeclareMathOperator{\id}{id}
\DeclareMathOperator{\Aut}{Aut}

\DeclareMathOperator{\image}{Im}

\DeclareMathOperator{\cd}{cd}

\newcommand\norm[1]{\lVert#1\rVert}

\newcommand*{\claimproofname}{Proof}

\crefname{cond}{condition}{conditions}
\creflabelformat{cond}{#2#1\@#3}
\crefname{obs}{observation}{observations}
\creflabelformat{obs}{#2#1\@#3}

\usepackage[T1]{fontenc}
\usepackage[margin=4cm]{geometry}

\usepackage{comment}

\makeatletter
\newtheorem*{rep@theorem}{\rep@title}
\newcommand{\newreptheorem}[2]{%
\newenvironment{rep#1}[1]{%
 \def\rep@title{#2 \ref{##1}}%
 \begin{rep@theorem}}%
 {\end{rep@theorem}}}
\makeatother

\newcommand{\ssrank}{\operatorname{ssrank}}
\newcommand{\bs}{\backslash}

\newcommand{\defq}{\mathrel{\mathop:}=}
\newcommand{\im}{\operatorname{im}}

\newcommand{\res}{\operatorname{res}}

\newcommand{\Isom}{\operatorname{Isom}}
\newcommand{\End}{\operatorname{End}}

\newcommand{\BL}{\operatorname{BL}}

\newcommand{\rank}{{\mathrm{rank}}}

\newcommand{\Ext}{{\mathrm{Ext}}}

\newcommand{\loc}{\operatorname{loc}}

\newcommand{\pol}{\ensuremath{\mathrm{pol}}}

\newcommand{\abs}[1]{{\left\lvert #1\right\rvert}}

\newreptheorem{theorem}{Theorem}

\newcommand{\N}{\mathbb{N}}

\newcommand{\FA}{\mathrm{FA}}

\definecolor{shakedcomment}{rgb}{0, 0, 255}
\newcommand*{\sh}[1]{{\color{shakedcomment}Shaked: #1}}

\definecolor{saarcomment}{rgb}{255, 0, 0}
\newcommand*{\saar}[1]{{\color{saarcomment}Saar: #1}}

\newcommand*{\rom}[1]{\textcolor{OliveGreen}{Roman: #1}}

\begin{document}

\title[Fixed point property and cohomology of arithmetic groups]{Fixed point properties and cohomology of Banach representations of arithmetic groups}
\date{}
\author{Saar Bader}
\email{saar.bader@gmail.com}

\author{Shaked Bader}
\address{University of Oxford}
\email{shaked.bader@gmail.com}

\author{Uri Bader}
\address{University of Maryland and Weizmann Institute of Science}
\email{uri.bader@gmail.com}

\author{Roman Sauer}
\address{Karlsruhe Institute of Technology}
\email{roman.sauer@kit.edu}

\begin{abstract}
    We study fixed point theorems for actions of lattices of semisimple groups. They are deduced from vanishing results for the group cohomology of $L^p$-representations. 
    We show that for lattices in simple groups of higher rank, the cohomology with $L^p$-coefficients vanishes below the rank whenever there are no invariant vectors. As a corollary of the vanishing for $L^1$-coefficients, we obtain that every action on an acyclic simplicial complex of dimension lower than the rank has a finite orbit.
    This in particular proves a conjecture by Farb. The $L^p$-vanishing below the rank proves a conjecture by Gromov regarding $L^p$-cohomology of symmetric spaces.
\end{abstract}



\maketitle

\section{Introduction}

The motivation for this paper was to study fixed point theorems for semisimple groups and their lattices.  Our level of generality are 
semisimple groups over arbitrary local fields of characteristic 0, but in the introduction we focus mainly on semisimple Lie groups.

For concreteness, we will discuss first the classical case of $\SL_n(\mathbb{Z})$.

\begin{theorem}[\cite{BorelSerre:cohomologie}, \cite{Pettet-Souto}]
    The group $\SL_n(\mathbb{Z})$ has a proper action on a
    contractible simplicial complex of dimension $n\choose{2}$,
    but it does not have a proper action on a lower-dimensional contractible simplicial complex.
\end{theorem}

Weakening the properness demand, we ask the following question.

\begin{question} \label{ques:d}
    What is the minimal number $d$ for which $\SL_n(\mathbb{Z})$ has an action on a contractible $d$-dimensional simplicial complex such that all orbits are infinite?
\end{question}

Clearly $d\leq {n \choose 2}$, but can we find a non-trivial lower bound for $n\geq 3$?
Serre showed that for $n\geq 3$, the group $\SL_n(\mathbb{Z})$ satisfies property FA, that is, every action on a \emph{tree} has a fixed point. In particular, $d\geq 2$.
One can prove Serre's result by combining the following two theorems.

\begin{theorem}[Kazhdan, \cite{kazhdan:connection}] \label{thm:Kaz67}
    For $n\geq 3$, $\SL_n(\mathbb{Z})$ satisfies property (T).
\end{theorem}

\begin{theorem}[Watatani, \cite{watatani:property}] \label{thm:wat}
    Groups with property (T) satisfy Serre's property FA.    
\end{theorem}

By proving analogs of the above two theorems, we are able to show that $d\geq n-1$.

\begin{theorem} \label{thm:SLn}
Every action of $\SL_n(\mathbb{Z})$ on a contractible $(n-2)$-dimensional simplicial complex has a finite orbit. 
\end{theorem}

The exact value of $d$ in \cref{ques:d} remains unknown. In particular, it is yet unknown whether $\SL_3(\mathbb{Z})$ acts on a contractible $2$-dimensional simplicial complex with infinite orbits.

Property (T) has the following cohomological consequence, which is actually an equivalent definition.

\begin{theorem}[Bader-Gelander-Monod, \cite{badergelandermonod:BGM:fixedpoint}]
\label{thm:BGM}
    A countable group $\Gamma$ has property (T) if and only if for every measure space $X$, and any linear isometric action of $\Gamma$ on $L^1(X)$,  $H^1(\Gamma,L^1(X))=0$.
\end{theorem}

The following is our generalization of \cref{thm:Kaz67}.

\begin{theorem} \label{thm:1stversion}
    Let $\Gamma=\SL_n(\mathbb{Z})$. For every measure space $X$ and every linear isometric action of $\Gamma$ on $L^1(X)$ without non-trivial invariant vectors, we have
    \[
    H^i(\Gamma,L^1(X))=0 \qquad\text{for }0\leq i\leq n-2.
    \]
\end{theorem}

The following is our generalization of \cref{thm:wat}.

\begin{theorem} \label{thm:2ndversion}
    Let $\Gamma$ be a group and let $d\geq0$. If $H^i(\Gamma,V)=0$ for $0\leq i\leq d$ and every $L^1$-space $V$ on which $\Gamma$ acts by linear isometries without non-trivial invariant vectors, then 
    every simplicial action of $\Gamma$ on a contractible $d$-dimensional simplicial complex has a finite orbit.
\end{theorem}

We will elaborate on and extend \cref{thm:1stversion} below.
\cref{thm:2ndversion} will be proved in \cref{sec:actions}, see the more concrete \cref{thm:finite_orbit_main}.
Together they clearly imply \cref{thm:SLn}.

The above is our main narrative, spelled out for $\SL_n(\mathbb{Z})$. Next we expand, explain and put it in context.
We start by recalling a classical cohomological identification of property~(T).
Throughout the paper, we abbreviate \emph{locally compact second countable} to \emph{lcsc}.

\begin{theorem}[Delorme-Guichardet, \cite{delorme:cohomologie} and \cite{guichardet:cohomologie}] \label{thm:DG}
   An lcsc group $G$ satisfies property (T) if and only if $H^1(G,V)=0$ for every unitary representation~$V$ of~$G$.  
\end{theorem}

This theorem has seen many generalizations.
Preceding \cref{thm:BGM}, there was the following.

\begin{theorem}[Bader-Furman-Gelander-Monod, \cite{baderfurmangelandermonod:BFGM:property}]
    An lcsc group $G$ has property (T) if and only if $H^1(G,L^p(X))=0$ for every measure space $X$, every $1<p \leq 2$ and every linear isometric action of $G$ on $L^p(X)$.
\end{theorem}

The analogous theorem for $p>2$ does not hold for general groups with property (T), but it does for lattices in simple groups of higher rank. 

\begin{theorem}[Bader-Furman-Gelander-Monod, \cite{baderfurmangelandermonod:BFGM:property} and \cite{badergelandermonod:BGM:fixedpoint}] \label{thm:H1Lp}
    Let $\Gamma$ be a lattice in a simple group of rank at least 2.
    Then $H^1(\Gamma,L^p(X))=0$ for every measure space $X$, every $1 \leq p < \infty$ and every linear isometric action of $\Gamma$ on $L^p(X)$.
\end{theorem}

The cohomological definition of property (T) due to Delorme-Guichardet, \cref{thm:DG}, was extended to higher degrees by two of us. 

\begin{definition}[\cite{badersauer:higherkazhdanproperty}]
      An lcsc $G$ has property $(T_n)$ if for every unitary $G$-representation~$V$ without non-trivial invariant vectors, we have $H^i(G,V)=0$ for every $i\leq n$.  
\end{definition}

In the same paper, the following counterpart of Kazhdan's original result was proved.

\begin{theorem}[\cite{badersauer:higherkazhdanproperty}] \label{thm:BShigher}
    Let $G$ be a connected simple Lie group of rank $r$ and let $\Gamma<G$ be a lattice.
    Then $G$ and $\Gamma$ satisfy property $(T_{r-1})$.
\end{theorem}

One of our main results of this paper is the following joint generalization of \cref{thm:BShigher} and \cref{thm:H1Lp}.

\begin{theorem} \label{thm:LpHi}
    Let $G$ be a connected simple Lie group with finite center and let $\Gamma<G$ be a lattice. Denote $r=\rank (G)$ and assume $r\geq 2$. We have $H^i(\Gamma,L^p(X))=0$ for every $i\leq r-1$, for every measure space $X$, for every $1 \leq p < \infty$, and every linear isometric action of $\Gamma$ on $L^p(X)$ without non-trivial invariant vectors. Further, $H^r(\Gamma,L^p(X))$ is Hausdorff.
\end{theorem}

The proof will be given in \cref{sec: vanishing}, see \cref{cor:main_simple_lattice}. A variant that holds for semisimple groups is given in \cref{thm:mainGsemisimpleLp}.
We get the following generalization of \cref{thm:SLn}.

\begin{theorem} \label{thm:fixedsemisimple}
    Let $G$ be a connected semisimple Lie group with finite center and no compact factors. Assume $G$ has property (T). Let $\Gamma<G$ be an irreducible lattice.
    Then every action of $\Gamma$ on a contractible simplicial complex of dimension less than $\rank(G)$ has a finite orbit.
\end{theorem}

The case where $G$ is simple is now an obvious corollary of \cref{thm:LpHi} together with \cref{thm:2ndversion}. 
The semisimple case will follow from \cref{thm:mainGsemisimpleLp} plus some extra work, see \cref{cor:farb_for_semisimple}.

While focusing in this introduction mainly on Lie groups, we cannot help but stating the following result regarding lattices in $p$-adic groups.
For a discussion and a proof of the theorem below, see \cref{sec:actions} and \cref{thm:padicquestionagain}.

\begin{theorem} \label{thm:padicquestion}
    For a simple $p$-adic group $G$ of rank $r$ and a lattice $\Gamma<G$, the answer to the analog to \cref{ques:d} is $d=r$.
\end{theorem}

In \cite{farb:group} Farb defined the group property $\FA_n$ to be the property that every cellular action on an $n$-dimensional CAT(0) cell complex has a fixed point.
By taking the circumcenter of a finite orbit, we obtain the following. See \cref{cor:farb_for_semisimple}.

\begin{corollary}[Farb's Conjecture] \label{mcor:farb}
   Let $G$ be a connected semisimple Lie group of rank $r$ with finite center and  property (T). Then any irreducible lattice in $G$ satisfies property $\FA_{r-1}$.  
\end{corollary}

Some cases were known before.
Farb proved this conjecture for some \emph{non-cocompact} lattices and asked whether it holds in general.
For cocompact lattices in a \emph{product} of groups, the harmonic map method solves the question. See \cite{Monod-SRproducts} and \cite{Gelander-Karlsson-Margulis}.
Recently, Frączyk and Lowe made some significant progress toward the conjecture using the method of minimal surfaces, see
\cite{fraczyklowe:minimal}.
We mention in particular their work on cocompact lattices in the exceptional rank 1 group $F_4^{-20}$, which they showed to satisfy $\FA_2$.
These lattices were shown to be $(T_3)$ in \cite{badersauer:higherkazhdanproperty}.
Frączyk and Lowe conjectured that groups with property $(T_n)$ have property $\FA_n$, a conjecture that we find very appealing. Fraczyk and Lowe recently announced a proof of Farb's conjecture for $SL_n(\bbZ)$ using the technique of minimal surfaces. 
We conjecture in tandem the following, which implies Frączyk and Lowe's conjecture.
\begin{conjecture}
    If a group $\Gamma$ has property $(T_n)$ then $H^i(\Gamma,L^1(X))=0$ for every $i\leq n$, every measure space $X$, and every linear isometric action of $\Gamma$ on $L^1(X)$ without non-trivial invariant vectors and $H^{n+1}(\Gamma,L^1(X))$ is Hausdorff.
\end{conjecture}

Another application of (the semisimple version of) \cref{thm:LpHi} is the proof of the following conjecture by Gromov.
For background regarding $L^p$-cohomology, see \cite{Bourdon-Remy-CE}. 
Since for a symmetric space $X$ there exists a simply connected semisimple real algebraic group $G$, such that the $L^p$-cohomology of $X$ is isomorphic to $H^*\bigl(G,L^p(G)\bigr)$, Gromov's conjecture follows from \cref{cor:Lp_cohomology_alg}.

\begin{corollary}[Gromov's conjecture] \label{cor:Gromov}
If $G$ is a connected semisimple Lie group with finite center and of rank $r$, then for every $1<p<\infty$ the $L^p$-cohomology of the associated symmetric space vanishes in degrees $i<r$ and it is Hausdorff in degree $r$.
\end{corollary}

We note that Gromov's conjecture for $1<p\leq 2$ was proved recently by Stern using geometric methods, see \cite{Stern}.
For every $1<p<\infty$, some partial results, in particular regarding the 2nd cohomology, were obtained by L{\'o}pez Neumann, see \cite{lopezneumann2026vanishing}. 
For more on the topic, as well as some other aspects of the conjecture, see the work of Bourdon-R\'emy, \cite{Bourdon-Remy-QI}, \cite{Bourdon-Remy-NV} and \cite{Bourdon-Remy-CE}.

\subsection*{The structure of the paper}

Roughly, we follow three steps. 
Simple Lie groups have \emph{vanishing of cohomology with $L^p$-coefficients}.
This vanishing property passes down to lattices using an \emph{induction technique} introduced in~\cite{badersauer:higherkazhdanproperty}. 
Vanishing of cohomology with $L^1$-coefficients implies a \emph{fixed point property}.

The next section, \cref{sec:prelim}, deals with various preliminaries, and it could be possibly omitted on a first read,
maybe with the exception of \cref{subsec:cohomproduct},
which discusses the cohomology of products, giving a variation on a theme developed in \cite{Baderrosendalsauer:onthe} which is instrumental for our discussion.

In \cref{sec:shapiro} we discuss the machinery needed to reduce vanishing of cohomology results from semisimple groups to their lattices.
This part of the work is not novel: we use a technique that was already developed in \cite{badersauer:higherkazhdanproperty} in the setup of unitary representations.
However, we need to reproduce it in the generality of $L^p$-spaces.

The most novel parts, also the technical heart of the paper, 
are \cref{section:Measurable homological algebra} and \cref{sec: vanishing}.
We devote \cref{sec: vanishing} to the proof of vanishing of cohomology theorem for semisimple groups.
The proof will be carried by an induction argument, based on reduction to Levi subgroups and using results from \cref{subsec:cohomproduct}.
The Levi subgroups are all parametrized together in a structure called \emph{the opposition complex}.
For the induction, we will have to deal with this complex.
There are several ways to approach this problem.
We chose to do it via a theory of \emph{measurable homological algebra},
which we develop in \cref{section:Measurable homological algebra}. This section is of independent interest. The contribution of this theory to our main results is only in the proof of \cref{prop:mainGinductionstep}.

Lastly, we will develop our fixed point machinery in \cref{sec:actions}.
Its first half is a stand-alone discussion of ``vanishing of cohomology implies fixed point'', and in its second half we will discuss various applications, based on the cohomological vanishing theorem provided in earlier sections.

\subsection*{AI statement}
No AI tools were
used to develop the ideas, arguments, results and proofs in it, nor to draft the manuscript.
AI (ChatGPT 5.6 and Claude Opus~5) was used for proofreading and locating references.
The
authors take full responsibility for the results.

\subsection*{Acknowledgment}

We thank Michael Glasner, Yuval Gorfine, Alex Lubotzky and Izhar Oppenheim for many inspiring conversations. The authors would like to thank the Isaac Newton Institute for Mathematical Sciences, Cambridge, for support and hospitality during the program \emph{Operators, Graphs and Groups}, where part of the work on this paper was undertaken. This work was supported by EPSRC grant EP/Z000580/1.

\section{Preliminaries} \label{sec:prelim}

In this section we provide the setup used throughout the paper and discuss necessary preliminaries.
In the first, very short \cref{sec:semisimple_groups}, we give our conventions regarding semisimple groups.
In \cref{sec:FA} we discuss the categories of Banach representations we deal with,
$L^p$-spaces and L-embedded spaces.
In \cref{sec:contcohom} we recall the basics of continuous cohomology.
In \cref{subsec:cohomproduct} we give special attention to the cohomology of product of groups.
In particular, we prove \cref{action of C on cohomology of N} which will be instrumental later, in \cref{sec: vanishing}.
While this proposition is not entirely novel, its applicability to L-embedded spaces is new. 

\subsection{Semisimple groups} \label{sec:semisimple_groups}

We regard in this paper \emph{semisimple Lie groups} and their lattices, as well as certain generalizations thereof. 
Recall that every connected, center-free, semisimple Lie group coincides with the identity component of the real points of a real semisimple algebraic group, namely the automorphism group of its Lie algebra.


When we say that $G$ is a \emph{semisimple group} (without mentioning the word ``Lie''), we mean that $G \cong \prod_{i=1}^l G_i$, where for each $i$, as topological groups, $G_i\cong \bfG_i(k_i)$ for some local field $k_i$ and a $k_i$-algebraic group $\bfG_i$ which is semisimple.
The \emph{rank} of $G$ is defined to be the sum of the ranks of its factors, $\rank(G)=\sum\rank_{k_i}(\bfG_i)$.
We will say that $G$ is \emph{connected and simply connected} if the algebraic groups $\bfG_i$ are connected and simply connected.

We will also regard \emph{reductive groups}.
By this we mean a group $G \cong \prod_{i=1}^l G_i$ as above, where the corresponding algebraic groups $\bfG_i$ are assumed to be reductive.
By a \emph{reductive subgroup} of such a reductive group $G$ 
we mean a closed subgroup $H \cong \prod_{i=1}^l H_i$ where for every $i$, 
$H_i\cong \bfH_i(k_i)$ for some reductive $k_i$-algebraic group $\bfH_i<\bfG$.
If each $\bfH_i<\bfG_i$ is a $k_i$-Levi subgroup, we say that $H$ is a \emph{Levi subgroup} of $G$ respectively.

A reductive subgroup $H$ is said to be a \emph{semisimple subgroup} if each $\bfH_i$ is semisimple.
In particular, the \emph{semisimple part} of 
the reductive group $G$ is the semisimple subgroup $H<G$
where each $\bfH_i<\bfG_i$ is the commutator subgroup of the identity component.
We define the rank of $G$ as above, but we will be careful to call it the \emph{reductive rank} of $G$, distinguishing it from the \emph{semisimple rank} of $G$, denoted $\ssrank(G)$, which is the rank of its semisimple part.

By an \emph{almost simple factor} of a reductive group $G$ we mean an almost simple factor of any of the groups $G_i$.
A lattice in a reductive group, $\Gamma<G$, is said to be \emph{irreducible} if it is dense modulo each non-compact almost simple factor of $G$.
%
We note that for admitting an irreducible lattice, all $k_i$ need to be of the same characteristic and a certain compatibility must occur between the various $\bfG_i$, by Margulis Arithmeticity Theorem.


\subsection{Functional-analytic preliminaries} \label{sec:FA}

\subsubsection{\texorpdfstring{$L^p$}{}-spaces} 

In this subsection we set up the notation and give some preliminaries regarding $L^p$-spaces, for $1\leq p<\infty$.
Banach spaces and linear maps, unless otherwise specified, are over the complex numbers.
We typically denote measure spaces simply by $X$ or $Y$.
These are sets endowed with a $\sigma$-algebra and a corresponding positive and $\sigma$-additive measure.
We then denote by $L^p(X)$ the Banach space of complex valued (classes of) $L^p$-functions, defined up to null sets, endowed with the $L^p$-norm.
When regarding the real Banach space of real valued functions, we emphasize it by writing $L^p(X,\mathbb{R})$.

By an \emph{$L^p$-space} we mean a complex Banach space $V$ which is isometrically isomorphic to $L^p(X)$ for some measure space $X$.
Choosing such an isometry imposes on $V$ the extra structure of a \emph{complex  Banach lattice} which we now recall. Our main sources are \cite{Lindestrausstzafriri:classical} and \cite{schaefer:Banachlattices}.

The theory of Banach lattices deals primarily with ordered real Banach spaces.
Recall that a \emph{real Banach lattice} is a triple $(U,\|\cdot\|,\leq)$ where $(U,\|\cdot\|)$ is a real Banach space and $\leq$ is a partial order on $U$ such that $U_+=\{x\in U\mid 0\leq x\} \subset U$ is a cone, every $x,y\in U$ have a least upper bound $x\vee y$ and a greatest lower bound $x\wedge y$, and the \emph{Banach lattice compatibility axiom} $|x|\leq |y| \Rightarrow \|x\|\leq \|y\|$ is satisfied, where $|x|=x \vee (-x)$, see \cite[Definition 1.a.1]{Lindestrausstzafriri:classical} and \cite[Definition II.5.1]{schaefer:Banachlattices}.
It follows from the axioms that the positive cone $U_+$ is closed in $U$ and the order operations $\vee$, $\wedge$ and $|\cdot|:U\to U_+$ are uniformly continuous, see \cite[Proposition II.5.2]{schaefer:Banachlattices}.
The following fundamental result in the theory of real Banach lattices is sometimes called \emph{Krivine's functional calculus}.

\begin{theorem}[{\cite[Theorem 1.d.1]{Lindestrausstzafriri:classical}}] \label{thm:1.d.1}
    Given a real Banach lattice $U$ and $x_1,\ldots,x_n \in U$ there is a unique linear lattice morphism from the lattice of all positively homogeneous continuous functions on $\mathbb{R}^n$ to $U$ such that $p_i\mapsto x_i$, where $p_i$ is the projection on the $i$'th axis.
\end{theorem}

The process of \emph{complexification} of a real Banach lattice is given equivalently, but differently, in \cite[\S II.11]{schaefer:Banachlattices} and \cite[p. 43]{Lindestrausstzafriri:classical}. We now present it as in the latter reference.
Given a real Banach lattice $U$, we identify $\mathbb{C}\otimes_{\mathbb{R}} U$
with $U+iU$, using the standard identification $U\cong \mathbb{R}\otimes_{\mathbb{R}} U$.
In particular, we consider the cone of positive elements $U_+$ as a subset of $\mathbb{C}\otimes_{\mathbb{R}} U$ and we regard the latter as an ordered vector space accordingly.
We consider the absolute value operator $|\cdot|:\mathbb{C}\otimes_{\mathbb{R}} U\to U_+$ given by $x+iy\mapsto (x^2+y^2)^{\frac{1}{2}}$, where we interpret this formula by means of \cref{thm:1.d.1}, 
and we endow $\mathbb{C}\otimes_{\mathbb{R}} U$ with the norm $\|z\|_{\mathbb{C}\otimes_{\mathbb{R}} U}=\|(|z|)\|_U$.
We define the complexification $U_{\mathbb{C}}$ of $U$ to be the complex ordered Banach space thus obtained.

A \emph{complex Banach lattices}, by definition, is an ordered complex Banach space $(V,\|\cdot\|,\leq)$ which is isomorphic as such to the complexification of a real Banach lattice. We note that the corresponding real Banach lattice $V_{\mathbb{R}}$ is uniquely determined as the span of the positive cone of $V$, endowed with the restricted norm and order relation. Accordingly, the absolute value operator $|\cdot|:V\to V_+$ is canonically determined and could be regarded as a part of the complex Banach structure.

\begin{lemma} \label{lem:upc}
    An ultraproduct of complex (real) Banach lattices is a complex (real) Banach lattice.
\end{lemma}

\begin{proof}
    The fact that an ultraproduct of real Banach lattices is a real Banach lattice is standard and easy. It follows from the fact that the defining axioms are defined pointwise, thus they commute with ultraproducts.
    We note that also the linear lattice morphisms alluded to in \cref{thm:1.d.1} commute with ultraproducts, by their uniqueness property. It follows that the absolute value map $|\cdot|$ 
    agrees on the ultraproduct of the complexifications and the complexification of the ultraproduct, thus the complexification process commutes with ultraproducts. We conclude that an ultraproduct of complex Banach lattices is a complex Banach lattice.
\end{proof}

A real or complex Banach lattice is called an \emph{abstract $L^p$-space} if it satisfies the following axiom.
\[     x,y \geq 0,~x\wedge y=0 \quad \Longrightarrow \quad \|x+y\|^p=\|x\|^p+\|y\|^p. \]

We say that $X$ is a \emph{Lebesgue space} if it is isomorphic as a measure space to a standard Borel space endowed with a $\sigma$-finite measure.

\begin{theorem}[Kakutani's realization, cf. {\cite[Theorem 7]{Kakutani:concrete}}] \label{thm:kak}
For every measure space $X$, $L^p(X)$ is a complex abstract $L^p$-space and $L^p(X,\mathbb{R})$ is a real abstract $L^p$-space.
If $X$ is a Lebesgue space then $L^p(X)$ is separable.
Conversely, every complex (or real) abstract $L^p$-space is isometric and order isomorphic to some $L^p(X)$ (or $L^p(X,\mathbb{R})$), where $X$ is a measure space. If the abstract $L^p$-space is separable the space $X$ could be taken to be a Lebesgue space.
\end{theorem}

\begin{proof}
The first half of the theorem is straight forward.
The complex version of the second half follows from the real one,
and the real one is \cite[Theorem 1.b.2]{Lindestrausstzafriri:classical}.
The fact about separable $L^p$-spaces goes back to \cite{Bohnenblust}.
\end{proof}

By Kakutani's realization, we get that this class of Banach spaces is stable under taking ultrapowers.

\begin{lemma} \label{lem:Lp_closed_under_ultraproducts}
    For a fixed $p$, any ultraproduct of $L^p$-spaces is an $L^p$-space.
\end{lemma}

\begin{proof}
Realizing the $L^p$-spaces as abstract $L^p$-spaces, using the first part of \cref{thm:kak}, we get an ultraproduct of abstract $L^p$-spaces, which is a Banach lattice by \cref{lem:upc}.
For $x,y \geq 0$ such that $x\wedge y=0$ in the ultraproduct, 
we have corresponding representatives $(x_n)$ and $(y_n)$ satisfying for every $n$, $x_n,y_n \geq 0$ and satisfying $x_n \wedge y_n \to 0$ along the underlying ultrafilter.
Setting $x'_n=x_n-x_n \wedge y_n$ and $y'_n=y_n-x_n \wedge y_n$, 
the sequences $(x'_n)$ and $(y'_n)$ represent $x$ and $y$ as well.
Since $x'_n,y'_n \geq 0$ and $x'_n\wedge y'_n=0$,
we have that $\|x'_n+y'_n\|^p=\|x'_n\|^p+\|y'_n\|^p$,
thus also $\|x+y\|^p=\|x\|^p+\|y\|^p$.
We conclude that the ultraproduct is an abstract $L^p$-space,
and the lemma follows by the second part of \cref{thm:kak}.
\end{proof}

The following fact is used in~\cite{badersauer:higherkazhdanproperty} in the context of unitary representation. It holds true for $L^p$-spaces with the same proof. 
Note, however, it does not hold for arbitrary ultralimits, just for ultrapowers.

\begin{lemma}\label{lem:no_almost_invariants_ultrapowers}
Let $\Gamma$ be a discrete group. For each $1\leq p<\infty$, the class of $L^p$-spaces with linear isometric $\Gamma$-actions without almost invariant vectors is closed under taking ultrapowers. 
\end{lemma}

\begin{proof}
The ultrapower is again an $L^p$-space by \cref{lem:Lp_closed_under_ultraproducts}. Having no almost invariant vectors passes to ultrapower by~\cite[Lemma~35]{badersauer:higher:survey}. 
\end{proof}

We are interested in the group of linear isometries between $L^p$-spaces.
This group is better behaved when the space is separable, hence the following observation will be useful.

\begin{lemma} \label{lem:Lpcofinal}
    Let $V$ be an $L^p$-space and $V_0<V$ a separable subspace. Then there exists an intermediate separable subspace $V_0\subset V_1\subset V$ that is an $L^p$-space as well.
    Moreover, if $G$ is an lcsc group which acts by linear isometries on $V$, then $V_1$ could be chosen to be $G$-invariant.
\end{lemma}

\begin{proof}
    We will prove the equivariant statement. We choose a countable dense subset $D_0$ in $V_0$ and a countable dense subgroup $\Gamma<G$. 
    We regard $V$ as an abstract $L^p$-space.
    We let $D_1$ be the smallest set containing $D_0$ and closed under $\Gamma$, $\mathbb{Q}(i)$-linear combinations, $\Re$, and $|\cdot|$; it is countable.
    We let $V_1$ be the closure of $D_1$. Then $V_1$ is a $\Gamma$-invariant sublattice of $V$, hence a $G$-invariant abstract $L^p$-space.
    We are done by \cref{thm:kak}.
\end{proof}

\begin{lemma}\label{lem:Lp_equivariant_projection}
Let $1\leq p<\infty$. Let $G$ be an lcsc group. Let $G$ act by linear isometries on an $L^p$-space $V$. A closed separable $G$-invariant subspace $V_0\subset V$ is linearly isometric to an $L^p$-space if and only if it admits a $G$-equivariant contractive linear projection $V\to V_0$.
\end{lemma}

\begin{proof}
The cases $V_0=0$ and $p=2$ are clear. First assume that $V_0$ is linearly isometric to an $L^p$-space. An application of \cite[Theorem~4]{ando:contractive} yields a contractive projection $P\colon V\to V_0$. One caveat is in order. The paper \cite{ando:contractive} requires $L^p$-spaces to be realized
by $L^p(X,\mu)$ with a finite measure~$\mu$. Up to linear isometry, this is equivalent to requiring that~$\mu$ is $\sigma$-finite. 
We can reduce to this situation as follows. We write $V=L^p(X,\mu)$. The measure restricted to the common support $S$ of $V_0$ is $\sigma$-finite because $V_0$ is separable. Now apply \cite[Theorem~4]{ando:contractive} to $V_0\subset L^p(S,\mu)$ and extend the projection by zero on $L^p(X\backslash S,\mu)$. 

This projection is unique among those vanishing on $L^p(X\setminus S,\mu)$: for $p>1$, this follows from Lemma~1 and the subsequent observation in~\cite{ando:contractive}. For $p=1$, choose $h\in V_0$ with support $S$ by \cite[Lemma~3]{ando:contractive}. After conjugating by the isometry $f\mapsto hf$ from $L^1(S,|h|\,d\mu)$ onto $L^1(S,\mu)$, uniqueness follows from \cite[Lemma~4]{ando:contractive}, since the conjugated projections have the same range and fix the constant function $1$.

Since $G$ preserves $V_0$ and $S$, $gPg^{-1}$ has the same properties for every $g\in G$. So uniqueness implies equivariance.

Conversely, given a $G$-equivariant contractive projection $V\to V_0$, restrict it to $L^p(S,\mu)$, where $S$ is the common $\sigma$-finite support of $V_0$ as above. This is still a contractive projection onto $V_0$, so \cite[Theorem~4]{ando:contractive}, with the same reduction to a finite measure, implies that $V_0$ is linearly isometric to an $L^p$-space.
\end{proof}

From now on we will focus exclusively on separable $L^p$-spaces.

\begin{definition}
Let $X$ be a Lebesgue space endowed with a measure $\mu$.
We denote by $\Aut(X)$ the group of all measure class preserving automorphisms of $X$. We denote by $T< \mathbb{C}^\times$ the unit circle and by $L(X,T)$ the group of all classes of measurable maps from $X$ to $T$, defined up to null sets. $\Aut(X)$ acts on $L(X,T)$ by group automorphisms, and we denote by $\BL(X)=\Aut(X)\ltimes L(X,T)$ the corresponding semidirect product, which we call the \emph{Banach-Lamperti} group of $X$. 
For every $1\leq p <\infty$ we define the Banach Lamperti map $\Phi_p\colon \BL(X)\to \Isom(L^p(X))$ by applying pointwise multiplication for elements of $L(X,T)$ and assigning for $\phi\in \Aut(X)$ the isometry of $L^p(X)$ given by 
\[ f \mapsto \left(\frac{d\phi^*\mu}{d\mu}\right)^{\frac{1}{p}}\cdot f\circ \phi^{-1}. \]
\end{definition}

\begin{definition}
    The \emph{Mazur map} $M_{p_1,p_2}$ from the unit sphere of $L^{p_1}(X)$ to the unit sphere of $L^{p_2}(X)$ is defined by $f\mapsto \operatorname{sgn}(f) \cdot |f|^{\frac{p_1}{p_2}}$, where $\operatorname{sgn}(z)=z/|z|$ for $z\neq0$ and $\operatorname{sgn}(0)=0$.
\end{definition}

\begin{theorem}[Banach-Lamperti, Mazur] \label{thm:BLM}
Let $X$ be a Lebesgue space and endow the groups $\Isom(L^p(X))$ with the SOT topology.
The group $\BL(X)$ has a group topology such that the maps 
$\Phi_p$ are all topological group isomorphisms onto their images, which are closed, and $\Phi_p$ is actually surjective for $p\neq 2$.
For each pair $p_1,p_2$, the Mazur map $M_{p_1,p_2}$ is uniformly continuous and it conjugates the actions of $\BL(X)$ on the corresponding unit spheres via the  Banach-Lamperti maps, that is for every $g\in \BL(X)$, $M_{p_1,p_2}\circ \Phi_{p_1}(g)=\Phi_{p_2}(g)\circ M_{p_1,p_2}$.
\end{theorem}

\begin{proof}
It is clear that the maps $\Phi_p$ are injective
and that the image of $\Phi_2$ is closed.
We identify $\BL(X)$ with its image in $\Isom(L^2(X))$, thus endow it with a group topology, making $\Phi_2$ an isomorphism onto its image.
The surjectivity of $\Phi_p$ for $p\neq 2$ is the main point of the classical Banach-Lamperti Theorem, \cite[Theorem 3.1]{Lamperti} (see also the Remark at the end of \S3 regarding point realization).
The uniform continuity of the Mazur map is given in \cite[Theorem 9.1]{Benyamini-Lindenstrauss}.
The commutation relation was observed in \cite[Lemma 4.2]{baderfurmangelandermonod:BFGM:property}.
It follows that the maps $\Phi_p$ are all topological group isomorphisms onto their images, which are closed.
\end{proof}

In view of the above, from now on we consider the groups $\BL(X)$ as a topological group.
The following corollary of \cref{thm:BLM} was observed in \cite{baderfurmangelandermonod:BFGM:property}.

\begin{lemma} \label{lem:noaiBL}
Given a Lebesgue space $X$ and a continuous homomorphism $\pi\colon G\to \BL(X)$,
the various representations $\Phi_p\circ \pi$ of $G$ on $L^p(X)$, $1\leq p<\infty$ either all have almost invariant vectors or none of them has almost invariant vectors.
\end{lemma}

\begin{proof}
    This follows from the uniform continuity and the equivariance property of the Mazur map, see \cite[\S4.a]{baderfurmangelandermonod:BFGM:property}.
\end{proof}


\begin{definition} \label{def:noaiBL}
    Given a locally compact second countable group $G$ and a continuous homomorphism $\pi\colon G\to \BL(X)$ we say that $\pi$ has almost invariant vectors if for some (equivalently, for all) $p$, the corresponding representation $\Phi_p\circ\pi$ has almost invariant vectors.
\end{definition}

Let $X$ and $Y$ be two Lebesgue spaces, let $T<\mathbb{C}^\times$ denote the unit circle, as before.
For $1\leq p<\infty$, the pointwise product map $L^p(X)\times L^p(Y) \to L^p(X\times Y)$
gives rise to a well defined (SOT-)continuous homomorphism 
$\Isom(L^p(X))\times\Isom(L^p(Y))\to\Isom(L^p(X\times Y))$.
Similarly, pointwise multiplication gives a homomorphisms $L(X,T) \times L(Y,T) \to L(X\times Y,T)$, which together with the obvious homomorphism 
$\Aut(X) \times \Aut(Y) \to \Aut(X\times Y)$ gives rise to a homomorphism 
$\BL(X) \times \BL(Y) \to \BL(X\times Y)$.
We observe the compatibility of the above homomorphisms

\[
\begin{tikzcd}[column sep=huge, row sep=large]
\BL(X) \times \BL(Y)
    \arrow[r]
    \arrow[d]
&
\BL(X\times Y)
    \arrow[d]
\\
\Isom(L^p(X)) \times \Isom(L^p(Y))
    \arrow[r]
&
\Isom(L^p(X\times Y))
\end{tikzcd}
\]
and deduce that the homomorphism $\BL(X) \times \BL(Y) \to \BL(X\times Y)$ is continuous.
Given an lcsc group $G$ and continuous homomorphisms $\pi\colon G\to \BL(X)$ and $\rho:G\to \BL(Y)$ we denote by $\pi\otimes \rho:G \to \BL(X\times Y)$ the corresponding composition
$G \to \BL(X) \times \BL(Y) \to \BL(X\times Y)$.

\begin{lemma} \label{lem:spectral gap induction}
    Let $G$ be a non-compact simply connected simple group over a local field of characteristic 0,
    acting on a separable $L^p$-space $V$ by linear isometries.
    Let $X$ be a Lebesgue space endowed with a probability measure preserving action of $G$ and consider the corresponding linear isometric representation of $G$ on $L^p(X,V)$.
    Then the action of $G$ on $V$ has almost invariant vectors if and only if the action  of $G$ on $L^p(X,V)$ has almost invariant vectors.
\end{lemma}

\begin{proof}
$V$ embeds in $L^p(X,V)$ as the subspace of constant functions, so if $V$ has almost invariant vectors then also $L^p(X,V)$ does. We will assume that $V$ has no almost invariant vectors and argue to show that $L^p(X,V)$ has no almost invariant vectors.
In case $p=2$, $V$ is a unitary representation and $L^2(X,V)$ is isomorphic to the Hilbertian tensor product $L^2(X)\bar{\otimes} V$ which has no almost invariant vectors by \cite[Theorem 1.1]{Gorfine:spectral}.
We thus assume $p\neq 2$.
We use \cref{thm:kak} to find a Lebesgue space $Y$ such that $V\cong L^p(Y)$.
The $G$-action on $X$ gives a continuous homomorphism $\pi\colon G \to \Aut(X) < \BL(X)$.
\cref{thm:BLM} guarantees that the $G$ representation on $L^p(Y)$ is via a certain continuous homomorphism $\rho:G\to \BL(Y)$.
We observe that $L^p(X,V)\cong L^p(X\times Y)$ in $G$-equivariant way, where the action on the right hand side is via $\Phi_p\circ (\pi\otimes \rho)$. We note that $\Phi_2\circ \rho$ has no almost invariant vectors by \cref{lem:noaiBL}.
By the case $p=2$, $L^2(X,L^2(Y))\cong L^2(X\times Y)$ has no almost invariant vectors.
We conclude by \cref{lem:noaiBL} that $L^p(X,V)\cong L^p(X\times Y)$ has no almost invariant vectors as well.
\end{proof}

In case the action on $X$ is trivial, $G$ could be an arbitrary lcsc group.

\begin{lemma} \label{lem:amplification}
    Let $G$ be an lcsc group acting on a separable $L^p$-space $V$ by linear isometries.
    Let $X$ be a probability measure space (with no $G$-action) and consider the corresponding action of $G$ on $L^p(X,V)$.
    Then the action of $G$ on $V$ has almost invariant vectors if and only if the action  of $G$ on $L^p(X,V)$ has almost invariant vectors.
\end{lemma}

\begin{proof}
The proof is essentially the same as the proof of \cref{lem:spectral gap induction}, except for the case $p=2$, which is simpler.
In this case, for the Hilbert space $V$, $L^2(X,V)\cong L^2(X)\bar{\otimes} V$ is an amplification of $V$ and one has almost invariants if and only if the other does.
\end{proof}

\subsubsection{\texorpdfstring{$L$}{}-embedded Banach spaces}

Closed subspaces $V$ and $W$ of a Banach space $U$ are said to be \emph{$\ell^1$-direct complementary} if $U=V\oplus W$
and for $v\in V$ and $w\in W$ we have $||v+w||=||v||+||w||$.
In such a case we say that $V$ is an \emph{$\ell^1$-direct complement} and we observe it determines $W$ uniquely.

\begin{definition}
    A Banach space $V$ is \emph{$L$-embedded} if its image in its bidual $V^{**}$ is an $\ell^1$-direct complement.
    In such a case we denote by $V_s<V^{**}$ its unique $\ell^1$-direct complement.
\end{definition}

For a thorough discussion of this class of Banach spaces, see the manuscript \cite{harmandwernerwerner:Mideals}.
Reflexive Banach spaces are $L$-embedded in a trivial way. 
In particular, $L^p$-spaces for $1<p<\infty$ are $L$-embedded.
It turns out that this also holds for $p=1$.

\begin{proposition}
    Every $L^1$-space is L-embedded. 
\end{proposition}

\begin{proof}
    See \cite[Chapter IV, Example 1.1]{harmandwernerwerner:Mideals}.
\end{proof}

The following fixed point theorem will be useful for us.

\begin{theorem}{\cite[Theorem A]{badergelandermonod:BGM:fixedpoint}} \label{BGM}
Let $V$ be an L-embedded Banach space
and let $B\subset V$ be a non-empty bounded subset.
Then there is a point in $V$ which is fixed by every isometry of $V$ which preserves $B$. 
\end{theorem}

The class of $L$-embedded Banach spaces is not closed under taking subspaces or quotients, but in some specific situations.
The slightly technical \cref{lem:Limage} below gives one of them.
Before stating it we will need some preparation.

\begin{lemma} \label{lem:l1sums}
    An $\ell^1$ direct sum of L-embedded spaces is L-embedded and for L-embedded spaces $V_1,\ldots,V_n$ the identification 
    \[ (V_1 \oplus_1 \cdots \oplus_1 V_n)^{**}\cong (V_1)^{**} \oplus_1 \cdots \oplus_1 (V_n)^{**} \]
    restricts to 
    \[ (V_1 \oplus_1 \cdots \oplus_1 V_n)_s \cong (V_1)_s \oplus_1 \cdots \oplus_1 (V_n)_s. \]
\end{lemma}

\begin{proof}
    The statement is straightforward for finite direct sums. For the general statement, see \cite[Chapter IV, Proposition 1.5]{harmandwernerwerner:Mideals}.
\end{proof}

\begin{lemma} \label{lem:algA}
    Let $V$ be an L-embedded Banach space and let $A<\End(V)$ be the norm closed algebra generated by all linear, surjective isometries of $V$.
    Then the action of $A$ on $V$ extends to an action on $V^{**}$ which preserves $V_s$.
    For $m,n\in\mathbb{N}$, every matrix $T\in M_{n\times m}(A)$, viewed as an operator $V^n\to V^m$, where these are taken with the $\ell^1$ direct sum norm, extends to an operator $(V^n)^{**}\to (V^m)^{**}$ and the image of $(V^n)_s$ is contained in $(V^m)_s$.
\end{lemma}

\begin{proof}
    The first statement holds for each isometry of $V$, by the uniqueness of the $\ell^1$ complement $V_s$, and it extends to $A$ by linearity and continuity.
    The second statements follows by \cref{lem:l1sums} for elementary matrices, that is a matrix with a single non 0 entry, and extends by linearity.
\end{proof}

\begin{lemma} \label{lem:Limage}
    Let $V$ be an L-embedded Banach space and let $T\in M_{n\times m}(A)$ be as in \cref{lem:algA}, considered as an operator $V^n\to V^m$, where these are taken with the $\ell^1$-direct sum norm. 
    If the image of $T$ is closed in $V^m$ then this subspace is L-embedded.
\end{lemma}

\begin{proof}
We identify $(T(V^n))^{**}$ with the closed subspace $T((V^n)^{**})<(V^m)^{**}$,
so the statement follows from \cref{lem:algA} see \cite[Chapter IV, Theorem 1.2]{harmandwernerwerner:Mideals}.
\end{proof}

\subsection{Continuous group cohomology and its variants} \label{sec:contcohom}

\subsubsection{Continuous group cohomology}\label{subsub: group cohomology}
Let $G$ be a topological group. A \emph{(continuous) $G$-module} is an abelian topological group $V$ on which $G$ acts continuously by automorphisms. 
The group 
\[C^k(G,V)=C\bigl(G^{k+1}, V\bigr)^G,\] 
where $k\ge 0$, of \emph{(continuous) $k$-cochains} of $G$ with values in $V$ is the set of continuous $G$-equivariant maps $G^{k+1}\rightarrow V$, where $G$ acts on $G^{k+1}$ diagonally by left multiplication. 
It is a topological abelian group with respect to pointwise addition and the compact-open topology. 
We obtain a cochain complex $C^\ast(G,V)$ via the usual (homogeneous) \emph{differential} or \emph{coboundary map} 
\[\delta^k\colon C^{k-1}(G,V)\rightarrow C^k(G,V),~~\delta^k(c)(g_0,\dots,g_k)=\sum_{i=0}^k(-1)^ic(g_0,\dots,g_{i-1},g_{i+1},\dots,g_k).\]

The \emph{continuous cohomology} $H^*_c(G,V)$ of $G$ with coefficients in $V$ is the cohomology of $C^\ast(G,V)$. 
We will usually omit the subscript $c$, except in \cref{section:Measurable homological algebra}.
The topology on the continuous cohomology is not necessarily Hausdorff. 
Let $Z^k(G,V):=\ker\delta^{k+1}$ and $B^k(G,V)=\image\delta^{k}$ be the \emph{$k$-cocycles} and \emph{$k$-coboundaries} respectively. 
The cohomology $H^k(G,V)$ is  Hausdorff if and only if the space of coboundaries $B^k(G,V)$ is closed.
We denote by $\Bar{H}^k(G,V)$ the largest Hausdorff quotient of $H^k(G,V)$, that is, $\Bar{H}^k(G,V)=Z^k(G,V)/\overline{B^k(G,V)}$. 
In degree zero we have $H^0(G,V)=\Bar{H}^0(G,V)=\ker\delta^1=V^G$, the subgroup of invariants. We will usually omit the superscript of $\delta^k$ and just write $\delta$. 

The following variation of the Lyndon-Hochschild-Serre spectral sequence for continuous cohomology is essentially due to Blanc. 

\begin{theorem}\label{thm:partial_hochschild_serre}
Let $G$ be an lcsc group and let $N\lhd G$ be a closed normal subgroup. 
Let $V$ be a continuous Fr\'echet $G$-module and let $n\geq 0$. If $H^q(N,V)$ is Hausdorff for $0\leq q\leq n$, then there is a first-quadrant spectral sequence converging to $H^{p+q}(G,V)$ such that
\[
E_2^{p,q}\cong H^p\bigl(G/N,H^q(N,V)\bigr)
\qquad\text{for all }p\geq 0\text{ and }0\leq q\leq n.
\]
Here $H^q(N,V)$ carries its natural Fr\'echet $G/N$-module structure.
\end{theorem}

\begin{proof}
We use the construction in the proof of \cite[Theorem~9.1]{Blanc:surla}. The construction of the double complex and the identification of its total cohomology with $H^*(G,V)$ in \cite[\S\S9.2--9.5]{Blanc:surla} do not require the cohomology of $N$ to be Hausdorff. Thus the associated first-quadrant spectral sequence exists and converges in all degrees.
The Hausdorff assumption enters only in \cite[\S9.6.2]{Blanc:surla}, where Proposition~1.6 of that paper is used to identify the vertical cohomology. In degree $q$, this requires only that the space of $q$-coboundaries be closed: kernels commute with $L^1_{\mathrm{loc}}$, as do images with closed range and quotients by closed subspaces. Hence the calculation in \cite[\S\S9.6.2--9.6.4]{Blanc:surla}  for each $q\leq n$ can be done the same way and yields the identification of $E_2^{p,q}$, without any Hausdorff assumption in higher degrees.
\end{proof}

\subsubsection{Polynomial group cohomology}

Let now $G$ be a compactly generated lcsc group with a word length function $l$. Let $V$ be as before. We denote by $C_\pol(G^{k+1},V)$ the space of continuous maps $f\colon G^{k+1}\to V$ satisfying
\[
    \norm{f(g_0,\dots,g_k)}\leq C\bigl(1+l(g_0)+\dots+l(g_k)\bigr)^d
\]
for all $g_0,\dots,g_k\in G$, for some constants $C>0$ and $d\in\mathbb{N}$ depending on $f$. The equivariant maps form a subcomplex 
\[C_\pol^\ast(G,V)=C_\pol\bigl(G^{\ast+1},V\bigr)^G\]
of $C^\ast(G,V)$. Its cohomology is the \emph{continuous polynomial cohomology} $H^*_{c,\pol}(G,V)$. The inclusion of cochain complexes induces the \emph{comparison map} $H^*_{c,\pol}(G,V)\to H^*_c(G,V)$. These definitions are independent of the choice of word length. For discrete $G$, we omit the subscript $c$. See also~\cite[\S~6.1]{badersauer:higherkazhdanproperty}.

\subsubsection{On separability and reducedness}

The following lemma is from~\cite[Lemma~34]{badersauer:higher:survey} and is based on an ultraproduct technique, which has a long history in the context of Property~(T). 

\begin{lemma} \label{lem:uptrick_34}
    Let $n\in\bbN$ and let $\Gamma$ be a discrete group with property $FP_n(\mathbb{R})$.
    Let $\mathcal{V}$ be a class of Banach $\Gamma$-modules that is closed under taking ultrapowers. 
    If 
    $\bar H^n(\Gamma,V)=0$ for every $V\in\mathcal{V}$, then $H^n(\Gamma,V)=0$ and $H^{n+1}(\Gamma,V)$ is Hausdorff for every $V\in \mathcal{V}$. 
\end{lemma}

For each $1\leq p<\infty$, this applies to the class of $L^p$-spaces with linear isometric $\Gamma$-actions by \cref{lem:Lp_closed_under_ultraproducts}. By \cref{lem:no_almost_invariants_ultrapowers}, it also applies when the restriction to each subgroup in a fixed family is required to have no almost invariant vectors.

Ultraproducts of Banach spaces, which is a necessary tool for the results above, are not separable. However, it is sufficient to show the cohomological vanishing in the separable case. The following is a generalization of~\cite[Lemma~3.6]{badersauer:higherkazhdanproperty}.

\begin{lemma} \label{lem:separable reduction}
Let $G$ be an lcsc group acting continuously by linear isometries on an $L^p$-space $V$, where $1\leq p<\infty$. Let $n\in\N$. 
Then $H^n(G,V)=0$ if and only  $H^n(G,V_0)=0$ for every separable $G$-invariant $L^p$-subspace $V_0\subset V$.
Further, $H^n(G,V)$ is Hausdorff if and only if $H^n(G,V_0)$ is Hausdorff for every such $V_0$.
\end{lemma}

\begin{proof}
As in the proof of \cite[Lemma~3.6]{badersauer:higherkazhdanproperty}, every continuous cochain has separable range. By \cref{lem:Lpcofinal}, it therefore takes values in a separable $G$-invariant $L^p$-subspace of $V$. If cohomology vanishes for all such subspaces, every cocycle is consequently a coboundary.

Let $c\in\overline{B^n(G,V)}$ and choose such a subspace $V_0$ containing its range. By \cref{lem:Lp_equivariant_projection}, there is a continuous $G$-equivariant projection $P\colon V\to V_0$. It induces a continuous map on cochains commuting with $\delta$, so $c=Pc\in\overline{B^n(G,V_0)}$. If $H^n(G,V_0)$ is Hausdorff, then $c\in B^n(G,V_0)\subset B^n(G,V)$.

Conversely, for every such $V_0$, the projection induces a continuous retraction of the cohomology groups. Thus vanishing and Hausdorffness pass from $H^n(G,V)$ to $H^n(G,V_0)$.
\end{proof}

\subsection{Cohomology of products} \label{subsec:cohomproduct}
The next proposition is a variation of \cite[Theorem 2]{Baderrosendalsauer:onthe} for the case of isometric representations on $L$-embedded spaces rather than weakly almost periodic representations. This will be needed in \cref{sec: vanishing}.
Afterward we record a few corollaries to it regarding vanishing of cohomology with $L^p$-coefficients for products.

\begin{proposition} \label{action of C on cohomology of N}
    Let $G$ be an lcsc group with a linear isometric representation on an $L$-embedded space $V$. Let $N,C<G$ be commuting subgroups such that $V$ does not have non-zero $C$-invariant vectors. Let $n\in\N$. 
    
    If $N$ is discrete with property $FP_n(\bbR)$ and the cohomology $H^n(N,V)$ is Hausdorff, then the action of $C$ on $H^n(N,V)$ induced by the $C$-action on $V$ does not have non-zero $C$-invariant vectors.
\end{proposition}

\begin{proof}
    Let $F_\ast$ be a free $\bbR[N]$-resolution of $\bbR$ such that $F_i\cong \bbR N^{k_i}$ is finitely generated for $i\in\{0,\dots, n\}$. The cochain complex $\hom_{\bbR[N]}(F_\ast,V)$ computes the cohomology of $N$ with coefficients in $V$. 
    The $C$-action on $V$ acts by chain isomorphisms and descends to an action on cohomology.    
    The cochain group in degree $i\le n$ is isomorphic to $V^{k_i}$. Furthermore, the $C$-action on the cochain group corresponds to the diagonal action on $V^{k_i}$. 
    
    Let $[z]\in H^n(N,V)$ be a $C$-invariant vector. We will show that $[z]=0$. The $C$-action on the invariant affine subspace 
    \[z+\image(\delta_{n-1})\subseteq V^{k_n-1}\cong \hom_{\bbR[N]}(F_n, V)\]
    has bounded orbits since the action of $C$ on $V$ is linear and isometric. 
    The $(n-1)$-coboundaries, $\image(\delta_{n-1})$, is $L$-embedded by \cref{lem:Limage}. By \cref{BGM}, the $C$-action on $z+\image(\delta_{n-1})$ has a fixed point, which is necessarily zero because of $V^C=0$. Thus, $[z]=0$. 
\end{proof}

\begin{corollary}\label{product of two groups}
    Let $G$ be an lcsc group with a linear isometric representation on an $L$-embedded space $V$. We assume that $G=NC$ is the product of two commuting subgroups $N$ and $C$ such that $N$ is discrete and has property $FP_n(\bbR)$ and $C<G$ is closed and $V^C=0$. 

    If $H^k(N,V)=0$ for all $k\in\{0,\dots, n-1\}$ and $H^n(N,V)$ is Hausdorff, then 
    $H^k(G,V)=0$ for all $k\in\{0,\dots, n\}$.
\end{corollary}

\begin{proof}
    We apply the version of the Hochschild-Serre spectral sequence in \cref{thm:partial_hochschild_serre} associated to $N\lhd G$ and converging 
    to $H^*(G,V)$. In the range $(p,q)\in\N\times\{0,\dots, n\}$ we have 
    $E_2^{pq}\cong H^p(G/N,H^q(N,V))$. By assumption, $E_2^{pq}=0$ for $0\le q<n$. Since $C$ centralizes $N$, the usual $G/N$-action on $H^n(N,V)$, pulled back to $C$, acts on the coefficients alone: the conjugation action on $N$ is trivial.
    By \cref{action of C on cohomology of N} $E_2^{0n}=H^0(G/N,H^n(N,V))=0$, which implies the claim. 
\end{proof}

The following corollary for discrete groups is almost straight forward from \cref{product of two groups}, and follows the slogan ``product of $n$ property $T^{L^p}$ groups is a property $T_n^{L^p}$ group'', and furthermore, this still holds in a non-trivial product if one of the groups does not have property $T^{L^p}$.


\begin{corollary} \label{cor:products_of_discrete}
    Let $p\in [1,\infty)$.
    Let $\Gamma=\Gamma_1\times\cdots\times \Gamma_{n}$ be a product of discrete groups of type $FP_{m+n-1}(\mathbb{R})$. We assume that $H^k(\Gamma_1,V)=0$ for every $k\in\{0,\dots m\}$ and for every linear isometric $\Gamma$-action on an $L^p$-space $V$ such that the restriction to any factor $\Gamma_i$ does not have almost invariant vectors. 
    
    If $V$ is such a representation, then $H^k(\Gamma,V)=0$ for all $k\in\{0,\dots, m+n-1\}$ and $H^{m+n}(\Gamma,V)$ is Hausdorff. 
\end{corollary}

\begin{proof}
    It is enough to prove the vanishing part of the statement since the Hausdorff property follows from  \cref{lem:uptrick_34} and the fact that the class of $L^p$-spaces with no almost invariant vectors for the factors is closed under taking ultrapowers by \cref{lem:no_almost_invariants_ultrapowers}.
    The case $n=1$ is then tautological. 
    
    Assume the statement for $n$ factors. We prove it for $n+1$ factors. Let $N=\Gamma_1\times \cdots \Gamma_n$.
    By the induction hypothesis $H^k(N,V)=0$ for all $k\in\{0,\dots, m+n-1\}$ and $H^{m+n}(N,V)$ is Hausdorff. By the assumption on $\Gamma_1,\dots,\Gamma_n$, we have that $N$ has type $FP_{m+n}(\mathbb{R})$ and $\Gamma_{n+1}$ has no non-zero invariants. The claim follows from \cref{product of two groups} applied to $N$ and $\Gamma_{n+1}$.  
\end{proof}

\begin{example}
    If $\Gamma_1$ has Kazhdan's property (T) and $\Gamma_1, \dots, \Gamma_n$ are of type $FP_{n-1}(\mathbb{R})$, then they satisfy the assumptions of \cref{cor:products_of_discrete} for $m=0$ when $p>2$ by \cite[Theorem A(1)]{baderfurmangelandermonod:BFGM:property} and if $\Gamma_1,\dots,\Gamma_n$ and of type $FP_n(\mathbb{R})$, they satisfy the assumptions for $m=1$ and $1\leq p\leq 2$ by \cite[Theorem 1.3(2)]{baderfurmangelandermonod:BFGM:property} and \cite{badergelandermonod:BGM:fixedpoint} for $p=1$.

    For Kazhdan property (T) groups, for every $1\leq p<\infty$, the assumption of having no almost invariant vectors is equivalent to the assumption of having no non-trivial invariant vectors~\cite[\S~4.c. and Theorem A]{baderfurmangelandermonod:BFGM:property}.
\end{example}

For $p=2$, the following is known for the reduced $\ell^p$-cohomology for arbitrary products of infinite groups by the $L^2$-Künneth formula; see~\cite[Theorem~1.35(4)]{Lueck:L2}.
It is also known for $p\neq 2$ and $n=2$ in a more general setting \cite[Theorem(iii) and Corollary 3.3]{martinvalette:on}.

\begin{example}
    If $\Gamma=\Gamma_1\times \cdots \times \Gamma_n$ and $\Gamma_1,\dots,\Gamma_n$ are non-amenable and of type $FP_n(\mathbb{R})$, then they do not have almost invariant vectors in the regular action on $\ell^p(\Gamma)$. Thus, by \cref{cor:products_of_discrete}, $H^k(\Gamma,\ell^p(\Gamma))=0$ for $0\leq k\leq n-1$ and $1\leq p<\infty$.
\end{example}

We cannot expect vanishing of the $n$-th cohomology groups, since $H^n(F_2^n,\ell^p(F_2^n))\neq 0$ for all $1\leq p<\infty$.

\section{Passing from the ambient group to a lattice} \label{sec:shapiro}

In this section, we show how to transfer vanishing results from the ambient semisimple group to a lattice, which is not necessarily uniform. This is essentially contained in~\cite[\S\S~6.4--6.6]{badersauer:higherkazhdanproperty}, which is largely agnostic about the difference between unitary and $L^p$-representations. In the sequel, we revisit the results and the discussion in~\cite[\S\S~6.4--6.6]{badersauer:higherkazhdanproperty} and make the transfer to the setting of $L^p$-Banach representations explicit. Further, the notion of universally integrable lattice~\cite[Definition~6.18]{badersauer:higherkazhdanproperty} is replaced by the more general, natural (and cocycle-free) notion of \emph{polynomial lattice}. The latter is a lattice that is at most polynomially distorted in the ambient group and admits a fundamental domain on which the word length function is $L^p$-integrable for every $p>0$. See \cref{def: polynomial lattice}.

\subsection{\texorpdfstring{$L^p$}{}-induction of Banach representations}

Let $G$ be an lcsc group. Let $V$ be a Banach space. 
If $(X,\mu)$ is a measurable space we denote by $L^p(X, V)$ the space of 
classes of Bochner-measurable functions $X\to V$ such that $\int_X \Vert f(x)\Vert^p d\mu(x)<\infty$. 

The following definition is the specialization of~\cite[Proposition~1.2 and Section~8.2]{Blanc:surla} to the setting of Banach spaces.

\begin{definition}\label{def: lp induction}
    Let $H<G$ be a closed subgroup, let $V$ be a Banach space with a continuous linear isometric $H$-action, let $p\in[1,\infty)$, and let $dx$ be a left Haar measure on~$G$. The \emph{(local) $L^p$-induction} $I^p_{\loc,H}(V)$ of $V$ is the space of classes  of Bochner-measurable functions $f\colon G\to V$ such that
    \[
        \int_K\norm{f(x)}_V^p\,dx<\infty
    \]
    for every compact subset $K\subseteq G$, and satisfying the equivariance condition
    \[
        f(xh^{-1})=h\cdot f(x)
    \]
    for a.e.~$x\in G$ and every $h\in H$. It is a Fr\'echet $G$-module for the left $G$-action
    \begin{equation}\label{eq: induction formula}
        (g\cdot f)(x)=f(g^{-1}x).
    \end{equation}

    Suppose in addition that $G/H$ carries a non-zero finite $G$-invariant Radon measure~$\mu$. The \emph{$L^p$-induction} $I^p_H(V)$ is the subspace of $I^p_{\loc,H}(V)$ consisting of the functions with finite norm
    \[
        \norm{f}_p
        \defq
        \Bigl(\int_{G/H}\norm{f(x)}_V^p\,d\mu(xH)\Bigr)^{\frac1p}.
    \]
    This is a Banach space, and the action in~\eqref{eq: induction formula} restricts to a linear isometric $G$-action on it.

    We may abbreviate $I^p_H(V)$ and $I^p_{\loc,H}(V)$ to $I^p(V)$ and $I^p_{\loc}(V)$ whenever the group $H$ is clear from the context.
\end{definition}

\begin{lemma} \label{lem:cocoIp}
    Let $H<G$ be a closed cocompact subgroup such that $G/H$ carries a non-zero $G$-invariant Radon measure, let $1\leq p<\infty$, and let $V$ be a Banach space with a continuous linear isometric $H$-action. Then the inclusion $I^p_H(V)\to I_{\loc,H}^p(V)$ is an isomorphism of Fr\'echet $G$-modules.
\end{lemma}
\begin{proof}
For compact $G/H$, global and local $L^p$-induction coincide by~\cite[Chapter~III, \S4, n\textsuperscript{o}~4.5,
p.~211, before Proposition~4.6]{guichardet:cohomologie}.
The inclusion is continuous and bijective, and so a topological
isomorphism by the open mapping theorem.
\end{proof}

\begin{lemma} \label{lem:cocolatticeVsplits}
    Let $H<G$ be a closed subgroup such that $G/H$ carries a non-zero finite $G$-invariant Radon measure~$\mu$. Let $V$ be a separable Banach space with a continuous linear isometric $G$-action, and regard $V$ as an $H$-module by restriction. Then there is a $G$-equivariant isometric isomorphism
    \[
        I^p_H(V)\cong L^p(G/H,V),
    \]
    where $G$ acts on the right-hand side by
    $(g\cdot f)(xH)=g\cdot f(g^{-1}xH)$.
    In particular, $V$ embeds as a $G$-equivariantly complemented subspace of $I^p_H(V)$.
\end{lemma}

\begin{proof}
    After rescaling, we may assume that $\mu$ is a probability measure. The map 
    \[
        T\colon I^p_H(V)\longrightarrow L^p(G/H,V),
        \qquad (Tf)(xH)=x\cdot f(x)
    \]
    is well defined by the equivariance condition of $f$. 
    Since the $G$-action on $V$ is isometric, $T$ is an isometry. Its inverse is the
    $G$-equivariant linear map
    \[
        T^{-1}\colon L^p(G/H,V)\longrightarrow I^p_H(V),
        \qquad (T^{-1}F)(x)=x^{-1}\cdot F(xH).
    \]
    Under this identification, the desired embedding of $V$ is the embedding into $L^p(G/H,V)$ as constant functions, which is $G$-equivariantly split by integration. 
\end{proof}

\begin{lemma}\label{lem: Lp modules preserved under induction}
Let $H<G$ be a closed subgroup such that $G/H$ carries a non-zero finite $G$-invariant Radon measure, and let $1\leq p<\infty$. Let $V$ be an $L^p$-space with a continuous linear isometric $H$-action. Then $I^p_H(V)$ is an $L^p$-space with a continuous linear isometric $G$-action.
\end{lemma}

\begin{proof}
By \cite[Chapter~III, \S4, n\textsuperscript{o}~4.5,
pp.~210--211]{guichardet:cohomologie}, the induction
$I_H^p(V)$ is isometrically isomorphic as a Banach space to
$L^p(G/H,V)$. Since $V$ is an $L^p$-space, so is $L^p(G/H,V)$.
This realization carries the continuous linear isometric induced $G$-action.
\end{proof}

\subsection{Review of the classical Shapiro isomorphism}

The following topological Shapiro isomorphism is given by \cite[Theorem~8.7, p.~163]{Blanc:surla}.

\begin{theorem} \label{thm:shapiro}
Let $G$ be an lcsc group, let $H<G$ be a closed subgroup, and let $1\leq p<\infty$.
Let $V$ be a Banach space endowed with a continuous linear isometric $H$-action.
Then there is a natural topological isomorphism, called the \emph{(local) Shapiro isomorphism}, 
\[
    H^\ast\bigl(G,I_{\loc,H}^p(V)\bigr)\cong H^\ast\bigl(H,V\bigr)
\]
In particular, it induces an isomorphism of the corresponding reduced cohomology groups. 
\end{theorem}

For the (global as opposed to local) $L^p$-induction, finite covolume yields a comparison map, and cocompactness ensures that it is an isomorphism.

\begin{theorem} \label{thm:cocoshapiro}
Let $G$ be an lcsc group, let $H<G$ be a closed subgroup such that $G/H$ carries a non-zero finite $G$-invariant Radon measure, and let $1\leq p<\infty$.
Let $V$ be a Banach space endowed with a continuous linear isometric $H$-action.
The inclusion $I_H^p(V)\hookrightarrow I_{\loc,H}^p(V)$ and the local Shapiro isomorphism induce a natural continuous linear map
\[
    H^\ast(G,I_H^p(V))\longrightarrow H^\ast(H,V).
\]
If, in addition, $H$ is cocompact in $G$, this map is a topological isomorphism, as is the induced map on the corresponding reduced cohomology groups. 
\end{theorem}

\begin{proof}
The comparison map is the composition
\[
    H^i(G,I_H^p(V))\longrightarrow H^i(G,I_{\loc,H}^p(V))
    \xrightarrow{\cong} H^i(H,V).
\]
By \cite[Theorem~8.7, p.~163]{Blanc:surla}, the second map is a topological isomorphism.
If $H$ is cocompact, \cref{lem:cocoIp} shows that the first map is a topological isomorphism as well.
Passing to the Hausdorff quotients yields the corresponding topological isomorphism on reduced cohomology.
\end{proof}

Combining this with the split map $V\to I_H^p(V)$ from \cref{lem:cocolatticeVsplits} gives the following restriction statement.

\begin{corollary} \label{cor:restriction_is_injective_cocolattice}
Let $G$ be an lcsc group and let $H<G$ be a closed cocompact subgroup such that $G/H$ carries a non-zero $G$-invariant Radon measure.
Let $V$ be a Banach space endowed with a continuous linear isometric $G$-action, and regard it as an $H$-module by restriction.
Then the restriction maps $H^\ast(G,V)\longrightarrow H^\ast(H,V)$ are split injective. Similarly for the corresponding maps in reduced cohomology. 
\end{corollary}

\subsection{Comparison of the cohomology of the lattice and the ambient group}

Let $\Gamma$ be a finitely generated lattice in a compactly generated lcsc group~$G$. Let $K$ be a compact symmetric generating set of~$G$. Let $S$ be a finite generating set of~$\Gamma$. Let $d_G$ be the word metric of $G$ associated with $K\cup S$. Let $d_\Gamma$ be the word metric of~$\Gamma$. Obviously, we have the estimate 
\[ d_G(\gamma, \gamma')\le d_\Gamma(\gamma, \gamma') \]
for all $\gamma, \gamma'\in\Gamma$. In the next definition, which is a streamlined version of~\cite[Definition~6.18]{badersauer:higherkazhdanproperty}, we ask that the inclusion $\Gamma\hookrightarrow G$ is at most polynomially distorted. 

\begin{definition}\label{def: polynomial lattice}
	 Let $\Gamma$ be a finitely generated lattice in a compactly generated lcsc group~$G$. Let $\mu$ be a Haar measure on~$G$. Let $d_G$ be a word metric of~$G$, and let $d_\Gamma$ be a word metric of~$\Gamma$. 
	 We say that $\Gamma$ is a \emph{polynomial lattice} if two properties are satisfied. 
	 \begin{enumerate}
	 	\item There is $C\ge 1$ and $k\in\bbN$ such that 
	 		 \[ d_\Gamma(\gamma, \gamma')\le C\cdot d_G(\gamma, \gamma')^k. \]
	 	\item There is a measurable fundamental domain $X\subset G$ of $\Gamma$ such that for every $p>0$ we have 
	          \[ \int_X d_G(x,1)^pd\mu(x)<\infty.\]
	 \end{enumerate}
\end{definition}

\begin{remark}\label{rem: indifference of polynomial complex to metric}
A consequence of the polynomial distortion of a polynomial lattice $\Gamma<G$ is that it does not matter whether we define the cochain complex of polynomial cohomology $C_\pol^\ast(\Gamma^{\ast+1}, V)$ with respect to a word metric of~$\Gamma$ or with respect to the restriction of a word metric of~$G$. The different word metrics yield the same cochain complex. 
\end{remark}

\begin{remark}\label{rem: Dirichlet domain also polynomial}
Following ideas from~\cite[\S~2]{Shalom:rigidity}, we show that one can replace the witnessing fundamental domain~$X$ in \cref{def: polynomial lattice} by a Dirichlet fundamental domain. 

Let $l(\_)=d_G(\_, 1)$ be a word length function associated with a compact symmetric generating neighborhood of $1$. One should keep \cref{rem: indifference of polynomial complex to metric} in mind. 
Let $X$ be a measurable $\Gamma$-fundamental domain witnessing the polynomiality of $\Gamma$. We consider the measurable subset
\[Y_0=\bigl\{g\in G\mid l(g\gamma)\ge l(g) \ \ \forall\gamma\in \Gamma\bigr\}. \] 
The right $\Gamma$-translates of $Y_0$ cover~$G$: For every $g\in G$ 
    pick an element $\gamma_g\in\Gamma$ that minimizes $d(g^{-1},\cdot)$, then  $l(g\gamma_g)=d(g^{-1}, \gamma_g)\le d(g^{-1}, \lambda)=l(g\lambda)$ for all $\lambda\in \Gamma$, so $g\gamma_g\in Y_0$. Thus, there is a measurable subset $Y\subset Y_0$ that is a $\Gamma$-fundamental domain. Since $X$ and $Y$ are fundamental domains there is a countable partition 
        $Y=\bigcup_{\gamma\in\Gamma} Y_\gamma$ such that $\phi\colon Y\to X$ with $\phi(y)=y\gamma^{-1}$ for $y\in Y_\gamma$ is a measure preserving isomorphism. Hence 
        \begin{equation*}  
        \begin{split}
        \int_Y l(y)^pd\mu(y)&=\sum_\gamma \int_{Y_\gamma} l(y)^pd\mu(y)\le \sum_{\gamma	\in\Gamma}\int_{Y_\gamma} l(y\gamma^{-1})^pd\mu(y)\\ &
            = \sum_{\gamma\in\Gamma}\int_{\phi(Y_\gamma)} l(x)^pd\mu(x)
            = \int_X l(x)^pd\mu(x)<\infty 
            \end{split}
        \end{equation*} 
 for every $p>0$. So $Y$ enjoys the same integrability condition as~$X$.

Furthermore, we may assume that $Y$ contains a neighborhood of $1$: A sufficiently small neighborhood $U$ contained in the generating set satisfies $U^{-1}U\cap\Gamma=\{1\}$ and $U\subset Y_0$, so we may retain its points when choosing representatives in $Y_0$.

\end{remark}

The following theorem is a consequence of the reduction theory of Borel-Harish-Chandra-Behr-Harder. The witnessing fundamental domain can be taken to be the union of Siegel domains. See~\cite[\S~3]{Shalom:rigidity} for a discussion and references. 

\begin{theorem}\label{thm: reduction theorem}
An irreducible lattice in a semisimple group of higher-rank is polynomial. 
\end{theorem}

The next lemma corresponds to~\cite[Proposition~6.20]{badersauer:higherkazhdanproperty} for universally integrable lattices in the sense of~\cite[Definition~6.18]{badersauer:higherkazhdanproperty}. The proof is similar except for the discussion about fundamental domains at the beginning.

\begin{lemma}[{\cite{badersauer:higherkazhdanproperty}, Proposition~6.20}]\label{lem: larger resolution for polynomial cohomology}
Let $\Gamma$ be a polynomial lattice in a compactly generated lcsc group~$G$. Let $V$ be a Banach space with a linear isometric $\Gamma$-action. Then 
the restriction homomorphism 
\[C_\pol(G^{\ast+1}, V)^\Gamma\to C_\pol(\Gamma^{\ast+1}, V)^\Gamma=C_\pol^\ast(\Gamma,V)\]
is a chain homotopy equivalence. 
\end{lemma}

\begin{proof}
Let $l(\_)=d_G(\_,1)$ be a word length function associated with a compact symmetric generating neighborhood of $1$. Choose a Dirichlet fundamental domain $Y$ containing a neighborhood of $1$ as in \cref{rem: Dirichlet domain also polynomial}.

Let $\pi\colon G\to\Gamma$ be the $\Gamma$-equivariant map defined by $g^{-1}\pi(g)\in Y$. For $y=g^{-1}\pi(g)\in Y$, the Dirichlet property gives
\[
l(y)\le l\bigl(y\pi(g)^{-1}\bigr)=l(g^{-1}),
\qquad l(\pi(g))\le l(g)+l(y)\le 2l(g).
\]
By \cref{rem: indifference of polynomial complex to metric}, we may use the restriction of $l$ to define polynomial cochains on $\Gamma$. 
The remainder of the argument in \cite[Proposition~6.20]{badersauer:higherkazhdanproperty}, which uses the Bochner integral on Banach spaces and polynomial norm estimates but not the Hilbert space structure of the coefficients, applies unchanged. \qedhere

\end{proof}

\begin{theorem}\label{thm: surjective Shapiro}
Let $\Gamma$ be a polynomial lattice in an lcsc group~$G$. Let $V$ be a Banach space with a linear isometric $\Gamma$-action. Let $p\in\bbR_{\ge 1}$. 
We have the following commutative diagram. 
\begin{equation}\label{eq: shapiro diagram}
\begin{tikzcd}
 H_\pol^\ast(\Gamma, V) \ar[r]\ar[dd] & H^\ast(\Gamma, V)\ar[d, "\cong"]\\
 & H^\ast(G, I^p_{\loc}(V))\\
 H_{c,\pol}^\ast(G, I^p(V))\ar[r]
 & H^\ast(G, I^p(V))\ar[u]
\end{tikzcd}
\end{equation}
The horizontal maps are the comparison maps between polynomial and ordinary cohomology. The left vertical map is induced by the chain map 
\[ \phi^\ast\colon C_\pol\bigl(G^{\ast+1}, V\bigr)^\Gamma\to  C_\pol\bigl(G^{\ast+1}, I^p(V)\bigr)^G,~\phi^n(f)(g_0,\dots, g_n)(x)=f(x^{-1}g_0,\dots, x^{-1}g_n).\]
The right upper vertical isomorphism is the Shapiro isomorphism in ordinary cohomology.
The right lower vertical map is induced by the inclusion $I^p(V)\subset I^p_{\loc}(V)$. 

In particular, the comparison map for~$\Gamma$ factors over a comparison map for~$G$. 
\end{theorem}

\begin{definition}
    The left vertical map in the diagram of \cref{thm: surjective Shapiro} is referred to as the \emph{Shapiro homomorphism in polynomial cohomology}. We call the right vertical composition $H^\ast(G, I^p(V))\to H^\ast(\Gamma, V)$ the \emph{Shapiro restriction}. 
\end{definition}

\begin{proof}
The comparison map is induced by the inclusion 
\begin{equation}\label{eq: realization comparison map} 
C_\pol(G^{\ast+1}, V)^\Gamma\to C(G^{\ast+1}, V)^\Gamma.	
\end{equation}
This follows from \cref{lem: larger resolution for polynomial cohomology} and from the fact that
the restriction $C(G^{\ast+1},V)^\Gamma\to C(\Gamma^{\ast+1}, V)^\Gamma$ is a cohomology isomorphism since both complexes before taking invariants are strong relatively injective $\Gamma$-resolutions of~$V$. See~\cite[Lemme~4.1 on p.~200]{guichardet:cohomologie} and~\cite[Proposition~2.9]{Blanc:surla}. Furthermore, the 
inclusion 
\[ C(G^{\ast+1}, V)^\Gamma\to L^p_{\loc}(G^{\ast+1}, V)^\Gamma\]
also induces a cohomology isomorphism for a similar reason because each $L^p_{\loc}(G^{n+1}, V)$ is a relatively injective $\Gamma$-module by~\cite[Th\'eor\`eme~3.4 and Proposition~8.3]{Blanc:surla} and $L^p_{\loc}(G^{\ast+1}, V)$ is a strong $\Gamma$-resolution by~\cite[Proposition~3.2.1]{Blanc:surla}. 

Next we show that $\phi^\ast$ is well defined. Let $X$ be a fundamental domain witnessing the polynomiality of $\Gamma$, and let $f\in C_\pol(G^{n+1},V)^\Gamma$ satisfy $\norm{f(h_0,\ldots,h_n)}\le C(1+\sum_i l(h_i))^d$. Then
\begin{align*}
\norm{f(x^{-1}g_0,\ldots,x^{-1}g_n)}
&\le C\Bigl(1+(n+1)l(x)+\sum_i l(g_i)\Bigr)^d\\
&\le C(n+1)^d(1+l(x))^d\Bigl(1+\sum_i l(g_i)\Bigr)^d.
\end{align*}
For each $(g_0,\ldots,g_n)\in G^{n+1}$, the function $x\mapsto f(x^{-1}g_0,\ldots,x^{-1}g_n)$ is continuous and has separable range since $G$ is second countable. So it is Bochner measurable and locally $p$-integrable. Taking the $L^p$-norm over $X$ yields
\begin{equation*}
\norm{\phi^n(f)(g_0,\ldots,g_n)}_{I^p(V)}
\le C(n+1)^d
\Bigl(1+\sum_i l(g_i)\Bigr)^d\Bigl(\int_X(1+l(x))^{dp}\,d\mu(x)\Bigr)^{1/p}, 
\end{equation*}
which is finite by the choice of $X$. Moreover,
\[
\phi^n(f)(g_0,\ldots,g_n)(x\gamma)
=\gamma^{-1}\phi^n(f)(g_0,\ldots,g_n)(x).
\]
Thus $\phi^n(f)$ takes values in $I^p(V)$ and has polynomial growth. The underlying pointwise estimate provides an integrable majorant for the $p$-th powers of the norms, locally uniformly in $(g_0,\ldots,g_n)$, so continuity follows by dominated convergence. The $G$-equivariance of $\phi^n(f)$ is straightforward, and $\phi^\ast$ commutes with the differential.

Finally, by \cite[Proposition~8.6 and the proof of Theorem~8.7]{Blanc:surla}, the Shapiro isomorphism is implemented on the local $L^p$-resolutions by the same formula as $\phi^\ast$. Hence diagram commutes in cohomology.
\end{proof}

By \cref{thm: surjective Shapiro}, the comparison map for $\Gamma$ factors through the Shapiro restriction. Surjectivity of the composition implies surjectivity of the latter map. Hence we obtain: 

\begin{corollary} \label{thm:shapiro_for_polynomial_lattice}
Let $\Gamma$ be a polynomial lattice in an lcsc group~$G$. Let $V$ be a Banach space with an isometric linear $\Gamma$-action. Let $p\ge 1$. If the comparison map $H^i_\pol(\Gamma, V)\to H^i(\Gamma, V)$ is surjective for degrees $0\le i\le d$, then 
the Shapiro restriction 
\[ H^i\bigl(G, I^p(V)\bigr)\to H^i(\Gamma, V) \]
 is surjective for degrees $0\le i\le d$.
\end{corollary}

In~\cite[Proposition~6.14]{badersauer:higherkazhdanproperty} we showed that, if $\Gamma$ has polynomial higher filling functions in a certain range, then the comparison map is an isomorphism in that range. The proof is stated for unitary representations, but the Hilbert structure is never used, and the argument works verbatim for linear isometric representations on Banach spaces. As a consequence of the deep work of Leuzinger-Young on the polynomiality of filling functions of arithmetic lattices below the real rank, we obtain in~\cite[Corollary~6.16]{badersauer:higherkazhdanproperty} the following result. 

\begin{theorem}
Let $\Gamma$ be an irreducible lattice in a connected semisimple Lie group~$G$ of rank~$r\ge 2$ with a finite center and without compact factors. 
Let $V$ be a Banach space with a linear isometric $\Gamma$-action. 
Then the comparison map $H^i_\pol(\Gamma, V)\to H^i(\Gamma,V)$ is an isomorphism for every $0\le i<r$. 
\end{theorem}

\begin{corollary} \label{cor:shapiro_lemma_for_lattices}
	Let $\Gamma$ be an irreducible lattice in a connected semisimple Lie group~$G$ of rank~$r\ge 2$ with a finite center and without compact factors. 
Let $V$ be a Banach space with a linear isometric $\Gamma$-action. 
Then the Shapiro restriction 
\[ H^i(G, I^p(V))\to H^i(\Gamma, V) \]
is surjective for degrees $0\le i<r$. 
\end{corollary}

\section{Measurable homological algebra}\label{section:Measurable homological algebra}

In this section we develop the theory of measurable homological algebra.
This is a stand-alone section that could be of independent interest.
Our only, yet crucial, application of it will be in \cref{subsec:5.4}, necessitated by a lack of a topological version of von Heydebreck's theorem on the sphericity of the opposition complex. 

We start with the very foundations in \cref{subsec:GT} and a
discussion of free objects in \cref{subsection:free abelian topological groups}. In \cref{subsec:MGC} we define and discuss a notion of measurable cohomology of measurable groups.

\subsection{General theory} \label{subsec:GT}
In this subsection we develop the language of homological algebra and derived functors over measurable groups and modules. Many of our claims are proved very similarly to the analogous claims in abelian categories and their proof will be omitted. We recall that a measurable space is a set endowed with a $\sigma$-algebra of subsets. 
\begin{definition}[Measurable algebraic structures]
 A \emph{measurable group} is a measurable space with a group structure such that the multiplication and inverse maps are measurable.
 A \emph{measurable ring} is a measurable space $R$ with a ring structure such that $(R,+)$ is a measurable group and the multiplication map $R\times R\to R$ is measurable.
 A \emph{measurable $R$-module} over a measurable ring $R$ is an $R$-module $M$ equipped with a $\sigma$-algebra such that $(M,+)$ is a measurable group and the multiplication map $R\times M\to M$ is measurable.

A $\sigma$-algebra on a group is a \emph{group $\sigma$-algebra} if it makes that group a measurable group. Similarly, we define a \emph{ring $\sigma$-algebra} and an \emph{$R$-module $\sigma$-algebra} on an $R$-module where $R$ is a measurable ring.
\end{definition}

\begin{lemma}\label{group sigma-algebras}
    Let $R$ be a measurable ring and let $A$ be an $R$-module.
    \begin{enumerate}
        \item If $f\colon A\to B$ is an $R$-module map where $(B,\Sigma)$ is a measurable $R$-module, then $(A,f^{-1}(\Sigma))$ is a measurable $R$-module.
        \item The $\sigma$-algebra generated by a collection of $R$-module $\sigma$-algebras on $A$ is an $R$-module $\sigma$-algebra.
    \end{enumerate}
\end{lemma}
\begin{proof}
    \begin{enumerate}
        \item Equip $A$ with the pullback $\sigma$-algebra $f^{-1}(\Sigma)$. We will show that the addition $+_A:A\times A\to A$ is measurable. The other properties are shown similarly. If $X$ is a measurable space then a map $X\to A$ is measurable if and only if its composition with $f$ is measurable. So $+_A$ is measurable if and only if $f\circ+_A$ is measurable, but since $f$ is additive, $f\circ+_A$ is the composition $A\times A\xrightarrow[]{f\times f}B\times B\xrightarrow[]{+_B}B$ which is measurable.
        \item Let $(\Sigma_i)_i$ be a collection of $R$-module $\sigma$-algebras on $A$. The product $\prod_i(A,\Sigma_i)$ is then a measurable $R$-module. The diagonal $\Delta A\subseteq \prod_i(A,\Sigma_i)$ equipped with the subspace $\sigma$-algebra is then a measurable $R$-module, which is isomorphic to $A$ equipped with the $\sigma$-algebra generated by all $\Sigma_i$.\qedhere
    \end{enumerate} 
\end{proof}
If $R$ is a ring and $X$ is a set, we denote by $RX$ or $R[X]$ the free $R$-module on $X$.
\begin{definition}[Free measurable modules]
    Let $X$ be a measurable space and $R$ a measurable ring. A \emph{free measurable $R$-module} on $X$ is a measurable $R$-module $A$ with a measurable map $X\to A$ such that for every measurable $R$-module $B$, every measurable map $X\to B$ factors through a unique measurable $R$-module map $A\to B$.
\end{definition}

\begin{proposition}\label{existence of free modules}
    Let $X$ be a measurable space and $R$ a measurable ring. There exists a free measurable $R$-module on $X$, which is unique up to a unique measurable isomorphism commuting with the structure maps from $X$. Its underlying $R$-module is $RX$.
\end{proposition}
\begin{proof}
    Let $\Sigma$ be the $\sigma$-algebra generated by all $R$-module $\sigma$-algebras on $RX$ for which the inclusion $X\to RX$ is measurable. By \cref{group sigma-algebras}, $\Sigma$ is also an $R$-module $\sigma$-algebra, and $X\to (RX,\Sigma)$ is measurable. Let $f\colon X\to A$ be a measurable map where $A$ is a measurable $R$-module. It extends uniquely to an $R$-module map $\bar{f}\colon RX\to A$. Denote by $\Sigma_A$ the $\sigma$-algebra of measurable subsets of $A$. By \cref{group sigma-algebras}, $(RX,\bar{f}^{-1}(\Sigma_A))$ is a measurable $R$-module and $X\to (RX,\bar{f}^{-1}(\Sigma_A))$ is measurable. Thus $\bar{f}^{-1}(\Sigma_A)\subseteq\Sigma$, hence $(RX,\Sigma)\to A$ is measurable.
\end{proof}

\begin{lemma}\label{lem:disjoint_union}
    Let $R$ be a measurable ring and $X$ and $Y$ measurable spaces. There is a natural measurable $R$-module isomorphism $R[X\sqcup Y]\cong RX\times RY$.
\end{lemma}
\begin{proof}
    The measurable maps $X\to X\sqcup Y$ and $Y\to X\sqcup Y$ induce measurable $R$-module maps $RX\xrightarrow[]{i_X} R[X\sqcup Y]$ and $RY\xrightarrow[]{i_Y} R[X\sqcup Y]$. Their sum, $RX\times RY\xrightarrow[]{i_X+i_Y} R[X\sqcup Y]$ is an $R$-module isomorphism and it is measurable since it is the composition \[RX\times RY\xrightarrow[]{(i_X,i_Y)} R[X\sqcup Y]\times R[X\sqcup Y]\xrightarrow[]{+} R[X\sqcup Y].\] 
    We will show that the inverse map $R[X\sqcup Y]\to RX\times RY$ is measurable. It is enough to show that composition with the projection $RX\times RY\to RX$ is measurable (the second projection is dealt with similarly). This composition is the $R$-module map $R[X\sqcup Y]\to RX$ which on $X$ is the inclusion $X\to RX$ and on $Y$ is identically zero. Hence it is measurable on both $X$ and $Y$, so also on $X\sqcup Y$. Hence it is measurable by the definition of the free measurable $R$-module.
\end{proof}

\begin{definition}[Measurable tensor product]
    Let $R$ be a measurable ring, $A$ a measurable right $R$-module and $B$ a measurable left $R$-module. The \emph{measurable tensor product} of $A$ and $B$ over $R$ is a measurable abelian group $A\otimes_RB$ with a measurable $R$-bilinear map $A\times B\to A\otimes_R B$ which is universal for all such maps.
\end{definition}
\begin{proposition}
    The tensor product exists and is unique up to a unique isomorphism commuting with the structure map. Its underlying abelian group is the usual tensor product of modules $A\otimes_RB$.
\end{proposition}
\begin{proof}
    Uniqueness is standard. The tensor product is $A\otimes_RB$ with the $\sigma$-algebra generated by all group $\sigma$-algebras for which the map $A\times B\to A\otimes_R B$ is measurable. Arguing similarly to the proof of \cref{existence of free modules}, we show that this is indeed a measurable tensor product.
\end{proof}

\begin{proposition}\label{properties of measurable tensor product}
    Let $R$ be a measurable ring.
    \begin{enumerate}
        \item\label{item:measurable_tensor_additive} If $B$ is a measurable left $R$-module, then the tensor product $A\mapsto A\otimes_RB$ is an additive functor.
        \item\label{item:measurable_tensor_free} If $R$ is commutative, then for all measurable spaces $X$ and $Y$ there is a natural bijective measurable $R$-module map $R[X\times Y]\to RX\otimes_RRY$.
    \end{enumerate}
\end{proposition}
\begin{proof}
    The proof of \cref{item:measurable_tensor_additive} is the same as that of the analogous algebraic fact. For \cref{item:measurable_tensor_free}, we first note that since $R$ is commutative, $RX$ is a measurable $R$-bimodule so $RX\otimes_RRY$ makes sense and it is a measurable $R$-module. Since $X\times Y\to RX\times RY\to RX\otimes_R RY$ is measurable, the natural bijection $R[X\times Y]\to RX\otimes_RRY$ is measurable.
\end{proof}

\begin{remark}
    There is no reason the map $R[X\times Y]\to RX\otimes_RRY$ will be a measurable isomorphism in general. In \cref{free abelian measurable group on a standard space}, we will prove that it is the case for $R=\bbZ$ and for standard Borel spaces.
\end{remark}
For the rest of this subsection, we will tacitly assume that every $R$-module is a measurable $R$-module and that every $R$-module map is measurable.

We now introduce projective measurable modules. We define them to have lifting property only for surjections that admit measurable sections to ensure that free modules are projective.

\begin{definition}[Strong maps]
    A measurable homomorphism between measurable abelian groups $\varphi\colon A\to B$ is called \emph{strong} if the surjection $A\twoheadrightarrow\image\varphi$ admits a measurable section (not necessarily homomorphic).
\end{definition}
\begin{definition}[Projective modules]
    An $R$-module $P$ is called \emph{projective} if for every strong surjective $R$-module map $A\twoheadrightarrow B$, any $R$-module map $P\to B$ can be lifted to $A$.
\end{definition}
\begin{proposition}\label{enough projectives}\hfill
    \begin{enumerate}
        \item\label{existence of enough projectives} Every $R$-module is a strong quotient of a free $R$-module.
        \item An $R$-module is projective if and only if it is a direct summand of a free $R$-module.
    \end{enumerate}
\end{proposition}
\begin{proof}\noindent
    \begin{enumerate}
        \item Let $M$ be an $R$-module. Consider $RM$, the free $R$-module on the measurable space $M$. By definition, we have a measurable map $i\colon M\to RM$. Since $M$ is an $R$-module, the identity map $M\to M$ extends to an $R$-homomorphism $p\colon RM\to M$. It is clear that $i$ is a section of $p$.
        \item Let $P$ be projective. By the first part, there exists a strong surjection $F\twoheadrightarrow P$ from a free module $F$. By projectivity we have an $R$-splitting $P\to F$, hence $P$ is a direct summand of $F$. For the other direction, it is enough to show that every free module is projective. Let $X$ be a measurable space and consider the free module $RX$. By the universal property of $RX$, its projectivity is equivalent to the property that for every strong surjection $p\colon A\twoheadrightarrow B$ of $R$-modules, every measurable map from $X$ to $B$ can be lifted to $A$. This follows from the existence of a measurable section. \qedhere
    \end{enumerate}
\end{proof}

The proofs of \cref{prop:exact_projective_is_contractible}, \cref{prop:fund_lemma_of_homological_algebra} and \cref{prop:L_0F} are identical to the analogous standard homological algebra proofs, and we omit them.

\begin{proposition}\label{prop:exact_projective_is_contractible}
    Let $R$ be a measurable ring and let
    \[
    P_n\xrightarrow[]{\partial_n} P_{n-1}\xrightarrow[]{\partial_{n-1}}\dots\xrightarrow[]{\partial_1}P_0\to0
    \]
    be a strong exact sequence of projective $R$-modules. There are $R$-module maps $h_k\colon P_{k}\to P_{k+1}$ for $k<n$ such that $\partial_{k+1}h_k+h_{k-1}\partial_k=Id$. A similar statement holds for an infinite projective exact sequence
    \[\dots\to P_1\to P_0\to0.\]
\end{proposition}

\begin{proposition}\label{prop:fund_lemma_of_homological_algebra}
    Every $R$-module admits a strong projective resolution. If $M$ and $N$ are $R$-modules, and $P_\ast\to M$ and $Q_\ast\to N$ are strong projective resolutions, then any $R$-module map $M\to N$ extends to an $R$-chain map $P_\ast\to Q_\ast$ which is unique up to an $R$-chain homotopy. It follows that any two strong projective resolutions of an $R$-module are $R$-homotopy equivalent.
\end{proposition}

\begin{definition}[Left derived functors]
    Let $F$ be an additive functor from the category of measurable $R$-modules to an abelian category $\mathcal{A}$. We define the \emph{left derived functors of} $F$, $(L_nF)_n$, as follows. Given an $R$-module $M$, choose a strong projective resolution $P_\ast\to M$. We define $L_nF(M)$ to be the $n$-th homology group of the complex $FP_\ast$.
\end{definition}

It follows from \cref{prop:fund_lemma_of_homological_algebra} that the derived functors $L_nF$ are well defined and it is easy to see that they are additive functors.
\begin{definition}[Right exact]
    An additive functor from the category of measurable $R$-modules to an abelian category $\mathcal{A}$ is called \emph{right exact} if it takes strong right exact sequences to right exact sequences.
\end{definition}

\begin{proposition} \label{prop:L_0F}
    Let $F$ be an additive functor from the category of measurable $R$-modules to an abelian category $\mathcal{A}$. If $F$ is right exact then $L_0F\cong F$.
\end{proposition}

\begin{definition}[Measurable Ext functors]
    Let $M$ be an $R$-module. The hom functor $\hom_R(\,\cdot\,,M)$ is an additive functor from the category of $R$-modules to the opposite category of the category of abelian groups. Its left derived functors are denoted $\Ext^*_R(\,\cdot\,,M)$. 
\end{definition}
\begin{proposition}
    The hom functor $\hom_R(\,\cdot\,,M)$ is right exact. In particular, $\Ext^0_R\cong\hom_R$.
\end{proposition}
\begin{proof}
    Let $A\to B\to C\to0$ be a strong right exact sequence of $R$-modules. We need to show the exactness of
    \[
    0\to\hom_R(C,M)\to\hom_R(B,M)\to \hom_R(A,M).
    \]
    Exactness at $\hom_R(C,M)$ is clear. In order to show the exactness at $\hom_R(B,M)$, we need to show that an $R$-module map $C\to M$ is measurable if its composition with $B\to C$ is measurable. This follows from the existence of measurable sections. 
\end{proof}
\begin{lemma}[Horseshoe lemma]\label{horseshoe lemma}
    Let $0\to A\to B\to C\to 0$ be a strong short exact sequence of $R$-modules. There are strong projective resolutions $P_\ast\to A$, $Q_\ast\to B$, $J_\ast\to C$ and (augmented) chain maps $0\to P_\ast\to Q_\ast\to J_\ast\to0$ that are strong short exact sequence in every degree.
\end{lemma}
\begin{proof}
    The proof follows as in \cite[Lemma 2.2.8]{weibel:introduction}. It is easily verifiable that all maps are measurable and strong.
\end{proof}
\begin{proposition}\label{LES}
    Let $F$ be an additive functor from the category of measurable $R$-modules to an abelian category $\mathcal{A}$. If $0\to A\to B\to C\to 0$ is a strong short exact sequence of $R$-modules, then there exists a natural long exact sequence
    \[
    \dots\to L_2F(C)\to L_1F(A)\to L_1F(B)\to L_1F(C)\to L_0F(A)\to L_0F(B)\to L_0F(C)\to 0.
    \]
\end{proposition}

\begin{proof}
    This follows from \cref{horseshoe lemma}, as in \cite[Theorem 2.4.6]{weibel:introduction}.
\end{proof}

The following is a refined version of the fact that derived functors can be calculated using any acyclic resolution. This is a standard fact in the algebraic context.

\begin{proposition}\label{acyclic resolution}
    Let $R$ be a measurable ring and let $F$ be a right exact additive functor from the category of measurable $R$-modules to an abelian category $\mathcal{A}$. Let $P_n\to\dots P_0\to M\to 0$ be a strong exact sequence with $L_mF(P_k)=0$ for $0<m\leq n-k$. For every $k<n$, $L_kF(M)$ is isomorphic to the $k$-th homology of $FP_\ast$ and $L_nF(M)$ is a subquotient of $F(P_n)$.
\end{proposition}
\begin{proof}
    
    Let $I_k=\ker(P_k\to P_{k-1})$ and $I_{-1}=M$. For every $k\leq n$, $0\to I_k\to P_k\to I_{k-1}\to0$ is a strong short exact sequence. Thus, it has the associated long exact sequence given by \cref{LES}. From here the statement follows using standard dimension shifting arguments (as in \cite[Exercise 2.4.3]{weibel:introduction}).
\end{proof}

\subsection{Free abelian topological groups}\label{subsection:free abelian topological groups}
For a set $X$ we denote by $\bbZ X$ the free abelian group on $X$. We consider $X$ as a subset of $\bbZ X$. For a topological space $X$, Markov defined the free abelian topological group on $X$ \cite{markov:free}. This is a group topology on $\bbZ X$ which makes the inclusion $X\to\bbZ X$ continuous and which is universal among all pairs of an abelian topological group $A$ together with a continuous map $X\to A$. The topology on $\bbZ X$ can be defined as the topology generated by all group topologies on $\bbZ X$ that make the inclusion $X\to\bbZ X$ continuous.

We will focus on the class of $k_\omega$-spaces. We recall that a $k_\omega$-space is a space $X$ admitting an increasing sequence of compact Hausdorff subsets $K_1\subseteq K_2\subseteq\dots$ such that $X=\bigcup_nK_n$ and it has the direct limit topology. If we write ``$X=\bigcup_nK_n$ is a $k_\omega$-space'', we mean that $(K_n)$ satisfies the above properties. Every $k_\omega$-space is Hausdorff. For a survey of $k_\omega$-spaces see \cite{franklin:surveyk_w}.
\begin{proposition}[\cite{mack:free}]\label{free abelian topological group}
    If $X=\bigcup_nK_n$ is a $k_\omega$-space, then $\bbZ X$ is a $k_\omega$-space and the topology on $\bbZ X$ is the strongest topology on $\bbZ X$ making the maps
\begin{equation}\label{defining maps for ZX}
\begin{split}
    K_n^{m+k} &\to\bbZ X,  \\
(x_1,\dots,x_m,y_1,\dots,y_k)&\mapsto \sum_{i=1}^m x_i-\sum_{j=1}^k y_j
\end{split}
\end{equation}
continuous, for all natural numbers $k,m,n$.
\end{proposition}
\begin{proof}
    See Theorem 1, its proof and Corollary 1 of \cite{mack:free}.
\end{proof}

\begin{lemma}\label{metrizable quotient}
    A Hausdorff quotient of a compact metrizable space is compact metrizable.
\end{lemma}
\begin{proof}
    We recall that a compact Hausdorff space $X$ is metrizable if and only if the Banach space $C(X)$ is separable.
    Let $X\to Y$ be a quotient map when $X$ is compact and metrizable and $Y$ is Hausdorff. In this case, $Y$ is also compact and $C(Y)$ embeds isometrically in $C(X)$. The separability of $C(X)$ implies the separability of $C(Y)$, which implies the metrizability of $Y$.
\end{proof}

\begin{proposition}\label{ZX is sigma-(compact metrizable)}
    Let $X=\bigcup_nK_n$ be a $k_\omega$-space. If each $K_n$ is metrizable, then $\bbZ X$ is a union of countably many compact metrizable spaces.
\end{proposition}
\begin{proof}
    For every $n,m,k$, consider the map $K_n^{m+k}\to\bbZ X$ defined in \cref{free abelian topological group}. Since $\bbZ X$ is a $k_\omega$-space, it is Hausdorff. Hence this map is a quotient map onto its image, so by \cref{metrizable quotient}, its image is compact metrizable. These images cover $\bbZ X$.
\end{proof}
The above property is useful since it implies the existence of Borel sections for continuous maps. 
    A topological space is called \emph{$\sigma$-(compact metrizable)} if it is a countable union of compact metrizable subspaces.

We will use the following, which is a corollary of the Kuratowski-Ryll-Nardzewski measurable selection theorem and appears as \cite[Corollary 2.3]{graf:selected}.
\begin{lemma}\label{lem:measurable_selection_compact}
    Let $X$ and $Y$ be compact metrizable spaces and $p\colon  X \to Y$ a continuous
surjective map. There exists a Borel-measurable section for $p$.
\end{lemma}

\begin{corollary}\label{lem:measurable_selection_sigma_compact}
    Let $X$ be a $\sigma$-(compact metrizable) space and let $Y$ be a Hausdorff space. If $p\colon  X \to Y$ is a continuous surjective map, then there exists a Borel-measurable section for $p$.
\end{corollary}
\begin{proof}
    Since $X$ is $\sigma$-(compact metrizable), we can write $X$ as a union of compact metrizable subsets $X=\bigcup K_n$. Since $Y$ is Hausdorff, for every $n$, $p(K_n)$ is compact and metrizable. 
    By \cref{lem:measurable_selection_compact}, there exists $s_n\colon p(K_n)\to K_n$ a section of $p|_{K_n}$ for every $n$. 
    For every $y\in Y$, define $s(y)=s_n(y)$ for $n$ minimal such that $y\in p(K_n)$. Clearly, $s$ is a section. It is measurable since it is measurable on each $p(K_n)-\bigcup_{m=1}^{n-1}p(K_{m})$ and $Y$ is the union of the closed sets $p(K_n)$.
\end{proof}
The following is an immediate corollary.
\begin{corollary}\label{section for defining maps for ZX}
    Let $X=\bigcup_nK_n$ be a $k_\omega$-space with each $K_n$ metrizable. Consider the map $p\colon \bigsqcup_{n,k,m}K_n^{k+m}\to\bbZ X$ defined on each $K_n^{k+m}$ by \cref{defining maps for ZX}. There exists a Borel section for $p$.
\end{corollary}
Using \cref{section for defining maps for ZX}, we can understand the free measurable abelian group on a standard Borel space. We will use the following terminology: if $X$ is a topological space then the \emph{Borel space of $X$} is the measurable space $X$ equipped with the $\sigma$-algebra of its Borel subsets.

\begin{theorem}\label{free abelian measurable group on a standard space}\hfill

    \begin{enumerate}
        \item\label{free abelian measurable group on a standard space 1} Let $X=\bigcup_nK_n$ be a $k_\omega$-space with each $K_n$ metrizable. The free measurable abelian group on the Borel space of $X$ is the Borel space of $\bbZ X$.
        \item\label{free abelian measurable group on a standard space 2} Let $X$ be a standard Borel space. The free abelian measurable group on $X$ is $\bbZ X$ endowed with the largest $\sigma$-algebra on $\bbZ X$ making the maps
        \begin{equation*}
            \begin{split}
                X^{k+m}&\to\bbZ X \\
                (x_1,\dots,x_m,y_1,\dots,y_k)&\mapsto \sum_{i=1}^m x_i-\sum_{j=1}^k y_j
            \end{split}
        \end{equation*}
        for $m,k\in\bbN$, measurable. Moreover, the associated map $\bigsqcup_{m,k}X^{m+k}\to\bbZ X$ has a measurable section and $\bbZ X$ is a standard Borel space.
        \item\label{free abelian measurable group on a standard space 3} Let $X$ be a standard Borel space, let $Y$ be any measurable space and let $A$ be a measurable abelian group. A map $\bbZ X\times Y\to A$ which is linear on each fiber $\bbZ X\times\{y\}$, $y\in Y$, is measurable if and only if its restriction to $X\times Y$ is measurable. 
        \item If $X$ and $Y$ are standard Borel spaces, then the natural map $\bbZ [X\times Y]\to\bbZ X\otimes\bbZ Y$ is an isomorphism. \label{free abelian measurable group on a standard space 4}
    \end{enumerate}
\end{theorem}
\begin{proof}
   \noindent \begin{enumerate}
        \item The inclusion $X\to\bbZ X$ is continuous and thus Borel measurable. Since $\bbZ X$ is $\sigma$-(compact metrizable), the Borel $\sigma$-algebra of $\bbZ X\times\bbZ X$ is the product of the Borel $\sigma$-algebras of the factors. Since $\bbZ X$ is a topological group, the addition $\bbZ X\times \bbZ X\to\bbZ X$ is continuous, hence Borel measurable, therefore $\bbZ X$ with its Borel $\sigma$-algebra is a measurable group.
        
        Let $f\colon X\to A$ be Borel measurable, where $A$ is a measurable abelian group. We need to show that the induced homomorphism $\bar{f}\colon\bbZ X\to A$ is Borel measurable. By \cref{section for defining maps for ZX}, the map $p\colon \bigsqcup_{n,k,m}K_n^{k+m}\to\bbZ X$ has a Borel section. Therefore, the measurability of $\bar{f}\circ p$ implies the measurability of $\bar{f}$. For every $n,k,m\in\bbN$, the restriction $\bar{f}\circ p|_{K_n^{k+m}}$ is the map $(x_1,\dots,x_m,y_1,\dots,y_k)\mapsto \sum_{i=1}^m f(x_i)-\sum_{j=1}^k f(y_j)$ which is measurable. Hence, $\bar{f}\circ p$ is measurable, as we wanted.
        \item A standard Borel space is the Borel space of a compact metrizable topological space. Let $X$ be such a topological space. By \cref{free abelian measurable group on a standard space 1}, the free abelian measurable group on the Borel space of $X$ is the Borel space of the free abelian topological group $\bbZ X$. That $\bbZ X$ is standard now follows from \cref{ZX is sigma-(compact metrizable)}. It follows from \cref{section for defining maps for ZX} that the map $\bigsqcup_{m,k}X^{m+k}\to\bbZ X$ has a measurable section, and this implies the first statement in \cref{free abelian measurable group on a standard space 2}.
        \item Let $f\colon \bbZ X\times Y\to A$ be linear on every fiber. Clearly, if $f$ is measurable then its restriction to $X\times Y$ is measurable. Assume that the restriction is measurable. Consider the map $p\colon \bigsqcup_{m,k}X^{m+k}\to\bbZ X$. Since $p$ has a measurable section, in order to show that $f$ is measurable it is enough to show that the composition \[(\bigsqcup_{m,k}X^{m+k})\times Y\xrightarrow[]{p\times I}\bbZ X\times Y\xrightarrow[]{f} A\] is measurable, and this can be proved exactly like the proof of \cref{free abelian measurable group on a standard space 1}.
        
        \item In order to show that $\bbZ [X\times Y]\to\bbZ X\otimes\bbZ Y$ is an isomorphism it is enough to show that every measurable map $X\times Y\to A$, where $A$ is a measurable abelian group, extends to a measurable bilinear map $\bbZ X\times \bbZ Y$. This follows from two applications of \cref{free abelian measurable group on a standard space 3}. \qedhere
    \end{enumerate}
\end{proof}

\subsubsection{Semi-simplicial spaces}
We recall the following standard definitions.

A \emph{semi-simplicial topological space} $X$ is a sequence $(X_n)_{n\in \mathbb{N}}$ of topological spaces with continuous \emph{face maps} $d^{n+1}_i\colon X_{n+1}\to X_n$ for $i=0\dots, n+1$ satisfying $d^n_id^{n+1}_j=d^n_{j-1}d^{n+1}_i$ whenever $i<j$. 
The space $X_n$ is called the space of \emph{$n$--simplices} of $X$.

    The \emph{simplicial chain complex} of a semi-simplicial topological space consist of the groups $\bbZ X_n$ with differential $\partial_n=\sum_i(-1)^id^n_i$. The \emph{simplicial homology}, $H_\ast(X)$, of the semi-simplicial space $X_\ast$ is the homology of this chain complex. Similarly, we define the augmented simplicial chain complex whose homology is the \emph{reduced homology} $\tilde{H}_\ast(X)$. 
    
    One should not confuse the notion of reduced homology here with another notion of reduced cohomology in the continuous group cohomology. 

\begin{proposition}\label{strong exactness of simlpicial spaces}
    Let $X=X_\ast$ be a semi-simplicial space such that for every $n$, $X_n=\bigcup_mK_{n,m}$ is a $k_\omega$-space and every $K_{n,m}$ metrizable. Let $n\in\N$ such that the reduced homology $\tilde{H}_k(X)$ vanishes for every $k< n$. For every $n$ consider $\bbZ X_n$ as the free measurable abelian group on $X_n$. The sequence of measurable abelian groups
    \begin{equation}\label{eq:res}
    \bbZ X_n\to\bbZ X_{n-1}\to\dots\bbZ X_0\to\bbZ\to0
    \end{equation}
    is strongly exact.
\end{proposition}
\begin{proof}
The vanishing of reduced homology is equivalent to the exactness of \cref{eq:res}. We only need to show the existence of measurable sections. Consider the free abelian topological groups $\bbZ X_k$. By \cref{free abelian measurable group on a standard space}, \cref{free abelian measurable group on a standard space 1}, the Borel space of $\bbZ X_k$ is the free abelian measurable group on the Borel space of $X_k$, therefore, we need to show that for every $k<n$, there exists a Borel section $\image\partial_{k+1}\to\bbZ X_{k+1}$. Note that the boundary map $\partial_{k+1}\colon \bbZ X_{k+1}\to\bbZ X_k$ is continuous. By \cref{ZX is sigma-(compact metrizable)}, $\bbZ X_{k+1}$ is $\sigma$-(compact metrizable) and by \cref{free abelian topological group}, $\bbZ X_k$ is a $k_\omega$-space hence Hausdorff. Therefore, the existence of Borel sections follows by \cref{lem:measurable_selection_sigma_compact}.
\end{proof}

An interesting corollary is the vanishing of measurable cohomology of such simplicial spaces. This corollary will not be used, but is of independent interest. If $X$ and $Y$ are measurable spaces, denote by $M(X,Y)$ the set of measurable maps $X\to Y$. If $Y$ is a measurable abelian group, then $M(X,Y)$ is a group by taking pointwise addition.

\begin{corollary}\label{cor: measurable maps complex}
    In the setting of \cref{strong exactness of simlpicial spaces}, for every measurable abelian group $A$, the measurable reduced cohomology of $X$ with coefficients in $A$, which is defined to be the cohomology of the complex
    \[0\to A\to M(X_0,A)\to M(X_1,A)\to\dots,\]
    with the usual coboundary maps, vanishes for $k<n$.
\end{corollary}
\begin{proof}
By \cref{strong exactness of simlpicial spaces} and \cref{prop:exact_projective_is_contractible}, the identity map of the partial complex
\[\bbZ X_n\to\bbZ X_{n-1}\to\dots\bbZ X_0\to\bbZ\to0\]
is chain-homotopy equivalent to the zero map. Hence, the same holds after taking $\hom(\,\cdot\,,A)$. We have $\hom(\bbZ X_k,A)\cong M(X_k,A)$, which gives the desired conclusion.
\end{proof}

\subsection{Measurable group cohomology} \label{subsec:MGC}
Throughout the rest of this section, $G$ denotes a measurable group whose underlying measurable space is standard (e.g. a Polish group).
\begin{lemma}\label{lem:measurable G-actions}
    The multiplication $G\times G\to G$ extends to a measurable ring structure on the free measurable abelian group $\bbZ G$. A measurable abelian group $A$ which is a $\bbZ G$-module is a measurable $\bbZ G$-module if and only if the action map $G\times A\to A$ is measurable. If $G$ acts on a standard Borel space $X$ such that the action map $G\times X\to X$ is measurable, then $\bbZ X$ is a measurable $\bbZ G$-module.
\end{lemma}
\begin{proof}
    It follows from \cref{free abelian measurable group on a standard space}, \cref{free abelian measurable group on a standard space 3} and \cref{free abelian measurable group on a standard space 4}.
\end{proof}
\begin{example}
     If $G$ is a Polish group and $A$ is a topological Polish $G$-module then $A$, endowed with its Borel $\sigma$-algebra, is a measurable $\bbZ G$-module.
\end{example}
Because of the above lemma, from now on we will sometimes write measurable $G$-module instead of measurable $\bbZ G$-module, and also measurable $G$-map, $G$-projective and so on.
\begin{definition}\label{def:measurable cohomology}
    The \emph{measurable cohomology} of a $G$-module $A$ is defined by
    \[H^*(G,A)=\Ext^*_{\bbZ G}(\bbZ,A).\]
\end{definition}
We recall that for measurable spaces $X$ and $Y$ we denote by $M(X,Y)$ the set of all measurable maps $X\to Y$. If $X$ and $Y$ are $G$-spaces we denote by $M_G(X,Y)$ the subset of $M(X,Y)$ of all $G$-equivariant maps.
\begin{lemma}\label{free ZG-modules}
    Let $X$ be a standard Borel space. The free measurable $\bbZ G$-module on $X$ is $\bbZ[G\times X]$.
\end{lemma}
\begin{proof}
    It follows from \cref{lem:measurable G-actions} that $\bbZ[G\times X]$ is a measurable $\bbZ G$-module. The composition $X\xrightarrow[]{x\mapsto(1,x)} G\times X\to\bbZ[G\times X] $ is measurable. Let $A$ be a measurable $\bbZ G$-module and let $f:X\to A$ be a measurable map. We need to show that the $G$-equivariant linear extension $\bar{f}:\bbZ[G\times X]\to A$ is measurable. The restriction $\bar{f}|_{G\times X}$ is the composition $G\times X\xrightarrow[]{Id\times f}G\times A\to A$ which is measurable. Therefore $\bar{f}$ is measurable.
\end{proof}
\begin{proposition}[Bar resolution]\label{measurable bar resolution}
    Let $A$ be a measurable $G$-module. The cohomology of $A$ is naturally isomorphic to the cohomology of the complex
    \[
    0\to M_G(G,A)\to M_G(G^2,A)\to M_G(G^3,A)\to\dots
    \]
     with the usual coboundary maps.
\end{proposition}
\begin{proof}
Consider the bar complex $\bbZ G^{\ast+1}$ with differential $\partial_\ast$ and its augmentation map $\partial_0\colon\mathbb{Z}G\to \bbZ$ (defined in the usual way). This is a complex of measurable $G$-modules. For $n\geq0$, the $G$-space $G^{n+1}$ is $G$-measurably isomorphic to $G\times G^n$ with a trivial action on the second factor. Hence, $\bbZ G^{n+1}$ is a free measurable $\bbZ G$-module. Define the measurable homomorphism $h_n\colon \bbZ G^n\to\bbZ G^{n+1}$ induced by the map $G^n\to G^{n+1}$ taking $(g_1,\dots,g_n)$ to $(1,g_1,\dots,g_n)$. We have that  $Id=h\partial+\partial h$ as chain complex maps. Thus, the augmented complex is exact and $h$ restricts to a measurable section $\image\partial_n\to\bbZ G^{n+1}$. Therefore, the bar complex is a strong projective resolution of the trivial $G$-module $\bbZ$. Hence, $H^*(G,A)$ is the cohomology of $\hom_{\bbZ G}(\bbZ G^{\ast+1},A)\cong M_G(G^{\ast+1},A)$.
\end{proof}

Using the bar resolution, we can compare measurable cohomology to continuous cohomology of locally compact groups. We denote by $H_c^*$ the continuous cohomology functor for locally compact groups. 

\begin{theorem}\label{comparison between continuous and measurable cohomology}
    Suppose $G$ is an lcsc group and the measurable structure on it is the Borel measurable structure.
    \begin{enumerate}
        \item\label{Moore's cohomology} If $A$ is a Polish $G$-module, then $H^*(G,A)$ is naturally isomorphic to Moore's measurable cohomology $H^*_\mathrm{m}(G,A)$.
        \item If $A$ is a separable Fréchet $G$-module, then there is a natural isomorphism $H^*(G,A)\cong H^*_c(G,A)$.
    \end{enumerate}
\end{theorem}

\begin{proof}
\begin{enumerate}
    \item Let $M_G(G^{\ast+1},A)$ be the measurable bar complex. For every $n$, let $L(G^{n+1},A)$ be 
    the group of measurable functions $G^{n+1}\to A$ modulo almost everywhere equality. 
    Denote by $L_G(G^{n+1},A)$ the $G$-equivariant functions. The usual coboundary maps give a complex $L_G(G^{\ast+1},A)$ and there is a natural map of complexes $M_G(G^{\ast+1},A)\to L_G(G^{\ast+1},A)$. Moore proved in \cite[Theorem 5]{moore:group} that this map induces an isomorphism in cohomology, and this cohomology is Moore's measurable cohomology $H^*_\mathrm{m}(G,A)$ (actually, Moore used the \emph{inhomogeneous} bar resolution, but it is easy to see that it is isomorphic to the homogeneous bar resolution which we use). By \cref{measurable bar resolution}, this gives a natural isomorphism $H^*(G,A)\cong H^*_\mathrm{m}(G,A)$.
    \item By \cite[Theorem A]{austinmoore:continuity} there is a natural isomorphism $H^*_c(G,A)\cong H^*_\mathrm{m}(G,A)$. Thus by \cref{Moore's cohomology}, $H^*_c(G,A)\cong H^*(G,A)$. \qedhere
\end{enumerate}
\end{proof}
If $H$ is a subgroup of $G$, we endow $G/H$ with the quotient $\sigma$-algebra. Our goal is to prove the following version of Shapiro lemma.
\begin{theorem}[Shapiro lemma]\label{measurable Shapiro lemma}
    Let $H$ be a subgroup of $G$. Assume that $H$ and $G/H$ are standard Borel spaces and that there exists a measurable section $G/H\to G$. For every $G$-module $A$, there is a natural isomorphism
    \[
    H^*(H,A)\cong\Ext_{\bbZ G}^*(\bbZ[G/H],A).
    \]
\end{theorem}
For the rest of this section we keep the setting of the theorem: $H$ is a subgroup of $G$, the spaces $G$, $H$ and $G/H$ are standard and there exists a measurable section $G/H\to G$. 
\begin{definition}[Induction of measurable actions]
    If $X$ is an $H$-measurable space then $H$ acts on $G\times X$ by $h(g,x)=(gh^{-1},hx)$. This commutes with the $G$-action on the first coordinate by left multiplication. The quotient of $G\times X$ by $H$ is denoted $G\times_H X$. We equip it with the quotient $\sigma$-algebra and with the induced $G$-action. It is called the induction of $X$ from $H$ to $G$.
\end{definition}
\begin{proposition}\label{prop:measurable induction of actions}
    Let $X$ be an $H$-measurable space.
    \begin{enumerate}
        \item There is a natural measurable isomorphism $G\times_H X\cong G/H\times X$ and the quotient map $G\times X\to G\times_H X$ admits a measurable section.
        \item The $G$-action on the induction $G\times_HX$ is measurable.
        \item\label{prop:measurable induction of actions 3} If the $H$-action on $X$ is restricted from a measurable $G$-action then there is a $G$-equivariant measurable isomorphism $G\times_H X\cong G/H\times X$ where $G$ acts diagonally on the right hand side. 
    \end{enumerate}
\end{proposition}
\begin{proof}
   \noindent \begin{enumerate}
        \item Let $\pi:G\to G/H$ be the quotient map and let $s:G/H\to G$ be a measurable section. Consider the map 
        \begin{align*}
            G\times X&\to G/H\times X\\ 
            (g,x)&\mapsto(\pi(g),[s(\pi(g))^{-1}g]\cdot x).
        \end{align*}
        It is measurable, and its fibers are the $H$-orbits of $G\times X$. This map admits a section $G/H\times X\xrightarrow[]{s\times Id}G\times X$. Hence, $G/H\times X$ has the quotient $\sigma$-algebra defined by this map, so it is isomorphic to $G\times_H X$, and the quotient $G\times X\to G\times_H X$ admits a section.
        \item Consider the commutative diagram
        \[
        \begin{tikzcd}
G\times G\times X \arrow[d] \arrow[r] & G\times X \arrow[d] \\
G\times(G\times_H X) \arrow[r]        & G\times_H X        
\end{tikzcd}
        \]
        where the horizontal maps are action maps and the vertical maps are quotient maps. We want to show the measurability of the lower horizontal map. Using a measurable section $G\times_H X\to G\times X$, it follows from the measurability of the other maps.
        \item Let $\pi:G\to G/H$ be the quotient map and let $s:G/H\to G$ be a measurable section.
        Consider the map 
        \begin{align*}
            G\times X&\to G/H\times X\\ 
            (g,x)&\mapsto(\pi(g),gx).
        \end{align*}
        It is measurable and $G$-equivariant, where $G$ acts only on the first coordinate of $G\times X$ and diagonally on $G/H\times X$, and its fibers are the $H$-orbits of $G\times X$. This map admits a section:
        \begin{align*}
            G/H\times X&\to G\times X\\ 
            (\pi(g),x)&\mapsto(s(\pi(g)),s(\pi(g))^{-1}x).
        \end{align*}
        It follows, as in the proof of the first point, that $G\times_H X$ is $G$-equivariantly isomorphic to $G/H\times X$.
        \qedhere
    \end{enumerate}
\end{proof}
\begin{proposition}\label{prop:measurable induction of modules}
    Let $X$ be an $H$-measurable space which is a standard Borel space. The $\bbZ G$-module $\bbZ G\otimes_{\bbZ H}\bbZ X$ is $G$-equivariantly and measurably isomorphic to $\bbZ [G\times_H X]$. In particular, $\bbZ G\otimes_{\bbZ H}\bbZ X$ is a measurable $\bbZ G$-module. Also, $\bbZ G\otimes_{\bbZ H}\bbZ X\cong\bbZ[G/H]\otimes\bbZ X$ as measurable abelian groups.
\end{proposition}
\begin{proof}
    The map $G\times X\to\bbZ G\otimes_{\bbZ H}\bbZ X$ taking $(g,x)$ to $g\otimes x$ is measurable and $H$-invariant, hence factors through a measurable map $G\times_H X\to\bbZ G\otimes_{\bbZ H}\bbZ X$, which defines a measurable homomorphism $\bbZ[G\times_H X]\to \bbZ G\otimes_{\bbZ H}\bbZ X$. This homomorphism is bijective and $G$-equivariant. We need to show that its inverse is measurable. 
    
    The composition $G\times X\to G\times_H X\to\bbZ[G\times_H X]$ is measurable, hence by \cref{free abelian measurable group on a standard space}, \cref{free abelian measurable group on a standard space 3}, the induced bilinear map $\bbZ G\times \bbZ X\to\bbZ[G\times_H X]$ is measurable. Therefore, the induced homomorphism $\bbZ G\otimes_{\bbZ H}\bbZ X\to\bbZ[G\times_H X]$ is measurable. This is easily seen to be the inverse of $\bbZ[G\times_H X]\to \bbZ G\otimes_{\bbZ H}\bbZ X$.
    
    Since $G\times_H X\cong G/H\times X$, it is a standard Borel space. Hence, by \cref{lem:measurable G-actions}, $\bbZ[G\times_H X]$ is a measurable $\bbZ G$-module. The last sentence of the claim follows using \cref{prop:measurable induction of actions} and \cref{free abelian measurable group on a standard space}, \cref{free abelian measurable group on a standard space 4}.
\end{proof}
\begin{proposition}\label{induction as adjoint functor}
    Let $X$ be an $H$-measurable space which is a standard Borel space. Let $A$ be a measurable $\bbZ G$-module. A $\bbZ H$-module map $\bbZ X\to A$ is measurable if and only if the induced $\bbZ G$-module map $\bbZ G\otimes_{\bbZ H}\bbZ X\to A$ is measurable.
\end{proposition}

\begin{proof}
    Let $f\colon \bbZ X\to A$ be a $\bbZ H$-module map and let $\bar{f}\colon \bbZ G\otimes_{\bbZ H}\bbZ X\to A$ be the induced $\bbZ G$-module map. The map $f$ is the composition of $\bar{f}$ with the measurable map $\bbZ X\to \bbZ G\otimes_{\bbZ H}\bbZ X$ taking $a$ to $1\otimes a$. Hence, if $\bar{f}$ is measurable, then so is $f$. Assume that $f$ is measurable. Conjugating with the isomorphism $\bbZ G\otimes_{\bbZ H}\bbZ X\cong\bbZ[G\times_H X]$ we obtain the map $\bbZ[G\times_H X]\to A$ which is defined on $G\times_H X$ by $[(g,x)]\mapsto gf(x)$, where $[(g,x)]$ is the image of $(g,x)\in G\times X$ in $G\times_H X$. Therefore, the restriction to $G\times_H X$ is measurable, and so it is measurable on $\bbZ[G\times_H X]$. Hence, $\bar{f}$ is measurable.
\end{proof}
\begin{corollary}\label{induction preserves freeness}
    If $X$ is a standard Borel space, then $\bbZ G\otimes_{\bbZ H}\bbZ [H\times X]\cong\bbZ [G\times X]$ as measurable $\bbZ G$-modules.
\end{corollary}
\begin{proof}
    By \cref{free ZG-modules}, $\bbZ[G\times X]$ is the free measurable $\bbZ G$-module on $X$. Thus, we want to show that every measurable map $X\to A$, where $A$ is a measurable $\bbZ G$-module, induces a measurable $\bbZ G$-module map $\bbZ G\otimes_{\bbZ H}\bbZ [H\times X]\to A$. Every measurable map $X\to A$ induces a measurable $\bbZ H$-module map $\bbZ[H\times X]\to A$ by \cref{free ZG-modules}, and this induces a measurable $\bbZ G$-module map $\bbZ G\otimes_{\bbZ H}\bbZ [H\times X]\to A$, by \cref{induction as adjoint functor}.
\end{proof}

\begin{lemma}\label{lem:standard Borel resolution}
    Let $A$ be a $G$-module which is a standard Borel space. There exists a strong resolution $P_\ast\to A$, where each $P_n$ is a free $\bbZ G$-module on a standard Borel space.
\end{lemma}
\begin{proof}
    Set $P_{-1}=I_{-1}=A$ and let $P_0=\bbZ G[I_{-1}]$. The canonical map $P_0\twoheadrightarrow A$ is a strong surjection. By \cref{free ZG-modules} and \cref{free abelian measurable group on a standard space}, $P_0\cong\bbZ[G\times A]$ is a standard Borel space.
    We continue by induction on $n$. Let $n\geq0$. Assume that we have a strong partial resolution 
    \[ P_n\to\cdots \to P_0\to A\to0,
    \]
    where $P_i$ is a free $\mathbb{Z}G$-module on a standard Borel space. 
    Let $I_{n}=\ker(P_{n}\to P_{n-1})$. Since $I_{n}$ is a measurable subset of the standard Borel space $P_n$, it is also a standard Borel space, by \cite[Corollary 13.4]{kechris:classical}.

    Let $P_{n+1}=\bbZ G[I_n]$. There is a strong $\bbZ G$-surjection $P_{n+1}\twoheadrightarrow I_n$. By \cref{free ZG-modules} and \cref{free abelian measurable group on a standard space}, $P_{n+1}\cong\bbZ[G\times I_n]$ is standard. Composing with the inclusion $I_n\hookrightarrow P_n$ gives the next differential. Thus, $P_{n+1}\to \cdots \to P_0\to A\to0$ is the required partial resolution.
\end{proof}
\begin{proof}[Proof of \cref{measurable Shapiro lemma}]
    Let $P_\ast\to\bbZ$ be a strong free resolution of $\bbZ$ as the trivial $\bbZ H$-module, where each $P_n$ is a free measurable $\bbZ H$-module on a standard Borel space, which exists by \cref{lem:standard Borel resolution}. Consider the augmented complex 
    \begin{equation}\label{induction resolution}
        \bbZ G\otimes_{\bbZ H} P_\ast\to \bbZ G\otimes_{\bbZ H} \bbZ.
    \end{equation}
    
    By \cref{prop:measurable induction of modules}, $\bbZ G\otimes_{\bbZ H} \bbZ\cong \bbZ[G\times_H\{*\}]$ as measurable $\bbZ G$-modules, and by \cref{prop:measurable induction of actions}, \cref{prop:measurable induction of actions 3}, $G\times_H\{*\}\cong G/H$ as $G$-measurable spaces. Therefore, $\bbZ G\otimes_{\bbZ H} \bbZ\cong \bbZ[G/H]$ as measurable $\bbZ G$-modules. 
    By \cref{prop:measurable induction of modules},  \cref{induction resolution} is isomorphic as an augmented complex of measurable abelian groups to $\bbZ G/H\otimes P_\ast\to\bbZ [G/H]$.
    By \cref{free ZG-modules}, every $P_n$ is a measurable free abelian group. Hence, \cref{prop:exact_projective_is_contractible} implies that the identity map of the complex $P_\ast\to\bbZ$ is chain-homotopy equivalent to the zero map. This property is preserved by additive functors, hence it also holds for $\bbZ [G/H]\otimes P_\ast\to\bbZ [G/H]$ which implies that \cref{induction resolution} is strong exact. 
    
    For every $n$, $P_n$ is a free $\bbZ H$-module hence by \cref{induction preserves freeness}, $\bbZ G\otimes_{\bbZ H} P_n$ is a free $\bbZ G$-module. We obtain that \cref{induction resolution} is a strong projective resolution of $\bbZ[G/H]$. Therefore $\Ext_{\bbZ G}^*(\bbZ[G/H],A)$ is the cohomology of $\hom_{\bbZ G}(\bbZ G\otimes_{\bbZ H} P_\ast,A)$. By \cref{induction as adjoint functor}, there is an isomorphism of complexes \[\hom_{\bbZ G}(\bbZ G\otimes_{\bbZ H} P_\ast,A)\cong\hom_{\bbZ H}(P_\ast,A).\] Thus, $\Ext_{\bbZ G}^*(\bbZ[G/H],A)$ is the cohomology of $\hom_{\bbZ H}(P_\ast,A)$ which is $\Ext^*_{\bbZ H}(\bbZ,A)$ that is, $H^*(H,A)$.
\end{proof}


\section{Vanishing of the cohomology of a semisimple group}
\label{sec: vanishing}
The main result of this section is \cref{thm:mainGsimpleLp}, which is a vanishing of cohomology result for simple group with respect to $L^p$-coefficients.
Note that this is a slightly more general version of \cref{thm:LpHi} from the introduction.


\begin{theorem} \label{thm:mainGsimpleLp}
    Let $G$ be a connected simply connected, simple group in characteristic 0 and assume $\rank(G)\geq 2$.
    Let $V$ be an $L^p$-space and assume that $G$ is acting on $V$ by linear isometries and $V^G=0$.
    Then for every $0<i<\rank(G)$, $H^i(G,V)=0$,
    and $H^{\rank(G)}(G,V)$ is Hausdorff.
\end{theorem}

Casually, we regard this result as \emph{vanishing below the rank} and \emph{Hausdorffness at the rank}.
Note that vanishing below the rank holds trivially for groups of rank 1, however these groups satisfy the Hausdorffness at the rank if and only if they satisfy property (T).
In our proof we will apply an induction argument to the Levi subgroups of~$G$ and their derived semisimple subgroup. The assumption on being simply connected carries over the semisimple part of a Levi subgroup by~\cite[Corollary~9.5.11]{conrad:reductive}. 
We will thus need a version of \cref{thm:mainGsimpleLp} which is applicable for general simply connected semisimple groups.


\begin{theorem} \label{thm:mainGsemisimpleLp}
    Let $G$ be a simply connected semisimple group in characteristic 0. If $V$ is an $L^p$-space such that $G$ acts on $V$ by linear isometries and every non-compact factor of $G$ does not admit almost invariant vectors, then  for every $0\le i<\rank(G)$, $H^i(G,V)=0$ and $H^{\rank(G)}(G,V)$ is Hausdorff.
\end{theorem}

This theorem is a generalization of \cref{thm:mainGsimpleLp}, as the following proof shows.

\begin{proof}[Proof of \cref{thm:mainGsemisimpleLp} $\Rightarrow$ \cref{thm:mainGsimpleLp}]
By \cref{lem:separable reduction}, we may and will assume that $V$ is separable. We realize $V$ as $L^p(X,\mu)$ for a $\sigma$-finite measure on a standard Borel space. Since $G$ has property (T) and $V^G=0$, \cite[Theorem~A(i)]{baderfurmangelandermonod:BFGM:property} implies that the action on $V$ has no almost invariant vectors. Thus $V$ satisfies the assumptions of \cref{thm:mainGsemisimpleLp} and so the claim follows.
\end{proof}




We will prove \cref{thm:mainGsemisimpleLp} by induction on the rank using three main steps.
First, using the following proposition, we reduce to showing the vanishing part of the conclusion.

\begin{proposition} \label{prop:mainGhausdorff}
If $G$ is a simply connected semisimple group in characteristic 0 such that $H^i(G,V)=0$ for every $0\le i<\rank(G)$ and every $L^p$-space $V$ on which $G$ acts without almost invariants for non-compact factors, then $H^{\rank(G)}(G,V)$ is Hausdorff for every such space and action.
\end{proposition}

Second, using the following proposition, we reduce to showing that for any $G$-representation as in \cref{thm:mainGsemisimpleLp}, the cohomology of all the proper Levi subgroups of $G$ vanishes up to the semisimple rank.

\begin{proposition} \label{prop:mainGinductionstep}
Let $G$ be a semisimple group and let $V$ be a separable Fréchet representation of $G$.
If for every proper Levi subgroup $L<G$ and for every $0\leq i\leq \ssrank(L)$ we have $H^i(L,V)=0$ then $H^i(G,V)=0$ for every $0\leq i<\rank(G)$.
\end{proposition} 

Lastly, by the induction hypothesis, we know that the semisimple part of each proper Levi satisfies the conclusion of \cref{thm:mainGsemisimpleLp}. The following proposition shows that this upgrades to the desired conclusion about the proper Levi subgroups.

\begin{proposition} \label{prop:mainLconditions}
    Let $G$ be a simply connected semisimple group.
    Let $V$ be an $L^p$-space on which $G$ acts by linear isometries without almost invariants for non-compact factors.
    Let $L<G$ be a proper Levi subgroup and let $S<L$ be its semisimple part.
    If $S$ satisfies the conclusion of \cref{thm:mainGsemisimpleLp}, then $H^i(L,V)=0$ for every $i\le\rank(S)$. 
    
    That is, if for every $L^p$-space $W$ on which $S$ acts by linear isometries such that there are no almost invariants for non-compact factors, we have $H^i(S,W)=0$ for $i<\rank(S)$ and $H^{\rank(S)}(S,W)$ is Hausdorff, then $H^i(L,V)=0$ for every $i\leq \rank(S)$. 
\end{proposition}

\begin{proof}[Proof of \cref{prop:mainLconditions,prop:mainGinductionstep,prop:mainGhausdorff}  $\Rightarrow$ \cref{thm:mainGsemisimpleLp}]
The theorem holds vacuously for rank 0 groups and it follows in general by induction.
By the induction assumption applied to the semisimple part of any proper Levi subgroup of $G$, which is simply connected when $G$ is by~\cite[Corollary~9.5.11]{conrad:reductive}, and \cref{prop:mainLconditions}, we get that $H^i(L,V)=0$ for every $L<G$ proper Levi subgroup, $V$ an $L^p$-space on which $G$ acts without almost invariants for non-compact factors and $i\leq\ssrank(L)$.

By \cref{prop:mainGinductionstep}, this implies $H^i(G,V)=0$ for $i<\rank(G)$ and for every such space and action $V$, if $V$ is separable. Using \cref{lem:separable reduction} we can remove the separability assumption. \cref{prop:mainGhausdorff} concludes the proof.
\end{proof}

In conclusion, in order to prove \cref{thm:mainGsimpleLp}, we are left to prove \cref{prop:mainGhausdorff,prop:mainGinductionstep,prop:mainLconditions}.
We will do this in the forthcoming three subsections, correspondingly.

\subsection{From vanishing below the rank to Hausdorffness at the rank}
\label{subsec: from vanishing to hausdorff}
In this subsection, we prove \cref{prop:mainGhausdorff}.
We use induction and restriction to reduce the statement to a cocompact lattice in $G$, where \cref{lem:great_implies_excellent_for_lattice} applies.

The group $G$ in \cref{prop:mainGhausdorff} decomposes as 
$G=\prod_iG_i$ with $G_i=\bfG_i(k_i)$ and $\operatorname{char}(k_i)=0$, where each $\bfG_i$ is almost $k_i$-simple.
For each factor $G_i$ in $G$, let $\Gamma_i<G_i$ be a cocompact lattice, whose existence is guaranteed by \cite[Theorem A]{BorelHarder:ExistenceCocompact},
and let $\Gamma<G$ be the corresponding product lattice.
For the rest of \cref{subsec: from vanishing to hausdorff} we maintain this setup. 

\subsubsection{Reduction to a reducible uniform lattice}

\begin{lemma} \label{lem:legitrestriction}
  Let $V$ be a separable $L^p$-space on which $G$ acts by linear isometries.
  For each non-compact factor $G_i$, the restriction to $G_i$ has almost invariant vectors if and only if the restriction to $\Gamma_i$ does.
\end{lemma}

\begin{proof}
It suffices to consider a single non-compact almost simple factor, which we denote by $G$, with cocompact lattice $\Gamma$.
If the $G$-action has almost invariant vectors, then so does its restriction to $\Gamma$. Conversely, assume that the $\Gamma$-action has almost invariant vectors.
By \cite[Theorem 6.1]{badergelander:equicontinuous}, $V^\Gamma=V^G$, so if $V^\Gamma\neq 0$ the conclusion follows. Otherwise, the coboundary map $\delta^1$ (see \cref{subsub: group cohomology}), and the existence of almost invariant vectors imply that its image is not closed. So $H^1(\Gamma,V)$ is not Hausdorff.
According to \cref{thm:cocoshapiro} and \cref{lem:cocolatticeVsplits}, $H^1\bigl(G,I^p_\Gamma(V)\bigr)$ is not Hausdorff and $I^p_\Gamma(V)\cong L^p(G/\Gamma, V)$. Thus the $G$-action on $L^p(G/\Gamma,V)$ has almost invariant vectors. From \cref{lem:spectral gap induction} we deduce that the $G$-action on $V$ has almost invariant vectors.
\end{proof}

\begin{lemma} \label{lem:legitinduction}
Let $V$ be a separable $L^p$-space on which $\Gamma$ acts by linear isometries. For each non-compact factor~$G_i$, the $\Gamma_i$-action on $V$ has almost invariant vectors if and only if the $G_i$-action on the $G$-induction $I^p_\Gamma(V)$ does.
\end{lemma}

\begin{proof}
    The absence of almost invariant vectors is equivalent to vanishing of $H^0$ and Hausdorffness of $H^1$. Therefore, for $G$ almost simple, the claim follows from the Shapiro isomorphisms in degrees $0$ and $1$ (\cref{thm:cocoshapiro}).
    We now consider a non-compact factor $G_i$ of $G$ and the corresponding lattice $\Gamma_i<G_i$. We consider $V$ as a $\Gamma_i$-representation and let $U$ be the corresponding $p$-induction to $G_i$.
    By the almost simple case, $V$ has $\Gamma_i$-almost invariant vectors if and only if $U$ has $G_i$-almost invariant vectors. We observe that as a $G_i$-representation, $I^p_\Gamma(V)\cong L^p(G'_i/\Gamma'_i,U)$, where $G'_i$ is the product of the factors other than $G_i$ and $\Gamma'_i<G'_i$ is the corresponding lattice.
    By \cref{lem:amplification}, this amplification has $G_i$-almost invariant vectors if and only if $U$ does, proving the claim.
\end{proof}

\subsubsection{The ultrapower argument and conclusion of the proof}

\begin{lemma} \label{lem:great_implies_excellent_for_lattice}
Assume that $H^n(\Gamma,V)=0$ for every linear isometric representation of $\Gamma$ on an $L^p$-space $V$ whose restriction to $\Gamma_i$ has no almost invariant vectors whenever $G_i$ is non-compact.
Then $H^{n+1}(\Gamma,V)$ is Hausdorff for every such representation.
\end{lemma}

\begin{proof}
    By \cref{lem:no_almost_invariants_ultrapowers}, taking ultrapowers preserves the assumptions on the representation. Since $\Gamma$ is a cocompact lattice, it is of type $F_\infty$.
    \cref{lem:uptrick_34} concludes.
\end{proof}

\begin{proof}[Proof of \cref{prop:mainGhausdorff}]
Set $r=\rank(G)$. 
Let $V$ be an $L^p$-space on which $\Gamma$ acts by linear isometries such that the restriction to $\Gamma_i$ has no almost invariant vectors whenever $G_i$ is non-compact. To show that $H^{r-1}(\Gamma,V)=0$, we may assume that $V$ is separable by \cref{lem:separable reduction}. By \cref{lem:legitinduction}, the restriction of $I^p(V)$ to each non-compact factor $G_i$ has no almost invariant vectors. The vanishing assumption for $G$ and Shapiro (\cref{thm:cocoshapiro}) give
\[
H^{r-1}(\Gamma,V)\cong H^{r-1}(G,I^p(V))=0.
\]
Thus \cref{lem:great_implies_excellent_for_lattice} implies that $H^r(\Gamma,V)$ is Hausdorff for every such representation.

Now let $W$ be a $G$-representation as in the proposition. Again, by \cref{lem:separable reduction}, we may assume that $W$ is separable. By \cref{lem:legitrestriction}, its restriction to $\Gamma$ satisfies the condition above, so $H^r(\Gamma,W)$ is Hausdorff. Then \cref{cor:restriction_is_injective_cocolattice} implies that $H^r(G,W)$ is Hausdorff as well.
\end{proof}

\subsection{Cohomological vanishing via the opposition complex} \label{subsec:5.4}

In this subsection we prove \cref{prop:mainGinductionstep}, which we recall here.
\begin{proposition*} 
Let $G$ be a semisimple group and let $V$ be a separable Fréchet representation of~$G$.
If for every proper Levi subgroup $L<G$ and for every $0\leq i\leq \ssrank(L)$ we have $H^i(L,V)=0$, then $H^i(G,V)=0$ for every $0\leq i<\rank(G)$.
\end{proposition*} 

The claim is void if $\rank(G)=0$, so we assume $\rank(G)\ge 1$.
This task will be achieved using the measurable cohomology of the opposition complex, which we will now describe. We use ideas from Monod \cite[\S~2.B]{Monod:onthe}. For background on buildings we refer to \cite{abramenkobrown:buildings}. 

In the notation of \cref{sec:semisimple_groups}, let $G=\prod G_i$ where $G_i\cong \bfG_i(k_i)$ is a semisimple group. Every $G_i$ acts on its spherical Tits building $\mathcal{T}^i$ which is of type $(W_i,S_i)$ where $W_i$ is the relative Weyl group of $G_i$, and of dimension $\rank_{k_i}(\bfG_i)-1$. We define the Tits building $\mathcal{T}$ of $G$ to be the join of the Tits buildings of the $G_i$'s. This is a spherical building of type $(W,S)$ where $W=\prod_iW_i$ and $S=\bigsqcup_iS_i$, and of dimension $\rank(G)-1$. 

Every vertex of $\mathcal{T}$ has a type, which is an element of $S$, and the $G$-action on $\mathcal{T}$ preserves the types. The vertices of a simplex in $\mathcal{T}$ have distinct types. Let $\mathcal{O}$ be the opposition complex of $\mathcal{T}$ as defined in \cite{vonheydebreck:homotopy}. Every simplex in $\mathcal{O}$ is a pair of opposite simplices (with opposite types) of $\mathcal{T}$. We define the type of a vertex of $\mathcal{O}$ as the type of its first coordinate. The $G$-action on $\mathcal{O}$ is type-preserving as well. In $\mathcal{O}$ the vertices of a simplex also have distinct types. Choose a linear order on $S$. 
Using the types, we get an ordering of the vertices of each simplex in $\mathcal{O}$. With face maps given by deleting the corresponding vertices, this defines the structure of a semi-simplicial set on~$\mathcal O$. We denote by $\mathcal{O}_n$ the set of $n$-simplices of $\mathcal{O}$. Since $G$ preserves the types, the group $G$ acts naturally on this semi-simplicial set. The geometric realization of the semi-simplicial set $(\mathcal{O}_n)_n$ is homeomorphic to the geometric realization of the simplicial complex $\mathcal{O}$. The following was proved by von Heydebreck.

\begin{theorem}[{\cite[Theorem~3.1]{vonheydebreck:homotopy}}]\label{sphericity of opposition complex}
    The homotopy type of the simplicial complex $\mathcal{O}$ is that of a wedge of spheres of dimension $\dim\mathcal{O}=\rank(G)-1$.
\end{theorem}
Next we identify the $G$-sets $\mathcal{O}_n$. It follows from standard algebraic group theory that for every $n\leq\rank(G)-1$, $\mathcal{O}_n$ is $G$-isomorphic to a finite disjoint union of $G$-sets of the form $G/L$ where $L$ is a Levi subgroup of $G$ of semisimple rank $\rank(G)-1-n$. We now give each $G/L$ its quotient topology from $G$. That makes $\mathcal{O}_n$ a $G$-space. On each orbit, the face maps are induced by inclusions of simplex stabilizers and hence are continuous. Thus $(\mathcal{O}_n)_n$ is a $G$-semi-simplicial space.

\begin{proof}[Proof of \cref{prop:mainGinductionstep}]
    In this proof we write $H_c^*$ for the continuous cohomology functor and $H_\mathrm{m}^*$ for the measurable cohomology functor defined in \cref{def:measurable cohomology}. Since $G$ is lcsc, it is a measurable group and a standard Borel space. The same holds for every Levi subgroup~$L$ of~$G$. Moreover, $G/L$ is a standard Borel space and by \cref{lem:measurable_selection_sigma_compact} there exists a measurable section $G/L\to G$. Each $G/L$ is a $k_\omega$-space, since it is locally compact, $\sigma$-compact and metrizable. The same holds for every $\mathcal{O}_{n}$. Hence, by \cref{sphericity of opposition complex} and 
    \cref{strong exactness of simlpicial spaces}, the sequence of measurable $G$-modules
    \[
    \bbZ \mathcal{O}_{\rank(G)-1}\to\dots\to\bbZ \mathcal{O}_{0}\to\bbZ\to0
    \]
    is strongly exact.
     
     Let $L$ be a proper Levi subgroup of semisimple rank~$r$, and let $i\leq r$. 
     By \cref{comparison between continuous and measurable cohomology} and the measurable Shapiro lemma, \cref{measurable Shapiro lemma}, $H^*_c(L,V)\cong H^*_\mathrm{m}(L,V)\cong  \Ext_{\bbZ G}^*(\bbZ G/L,V))$.

     Since $\mathcal{O}_n$ is a finite disjoint union of $G/L$ over $L\in\mathcal{L}$, a finite set of Levi subgroups satisfying $\ssrank(L)=\rank(G)-1-n$, \cref{lem:disjoint_union} and the additivity of the functor $\Ext^i_{\bbZ}(\,\cdot\,,V)$ give
     \[\Ext_{\bbZ G}^i(\bbZ\mathcal{O}_{n},V)=\bigoplus_{L\in\mathcal{L}} \Ext_{\bbZ G}^i(\bbZ [G/L],V).\]
     Thus, by assumption, 
     \begin{equation} \label{eq:vanishing}
         \Ext_{\bbZ G}^i(\bbZ\mathcal{O}_{n},V)=0 \text{ for }  n\leq\rank(G)-1\text{ and } i\leq\rank(G)-1-n.
     \end{equation}
     
     By \cref{acyclic resolution}, in degrees smaller than $\rank(G)-1$, $H^*_\mathrm{m}(G,V)=\Ext^*_{\bbZ G}(\bbZ,V)$ is the cohomology of $\Ext_{\bbZ G}^0(\bbZ\mathcal{O}_{\ast},V)$, and $H^{\rank(G)-1}_\mathrm{m}(G,V)=\Ext^{\rank(G)-1}_{\bbZ G}(\bbZ,V)$ is a subquotient of $\Ext_{\bbZ G}^0(\mathcal{O}_{\rank(G)-1},V)$. Both vanish by \cref{eq:vanishing}.
     By \cref{comparison between continuous and measurable cohomology}, we obtain that $H^i_c(G,V)\cong H^i_\mathrm{m}(G,V)=0$ for $0\le i\le\rank(G)-1$.
\end{proof}

\subsection{Cohomological vanishing for Levi subgroups}\label{subsec: cohomological vanishing}

In this subsection, we prove \cref{prop:mainLconditions}, which we recall here.

\begin{proposition*} 
    Let $G$ be a simply connected semisimple group.
    Let $V$ be an $L^p$-space on which $G$ acts by linear isometries such that the restriction to every non-compact almost simple factor has no almost invariant vectors.
    Let $L<G$ be a proper Levi subgroup and let $S<L$ be its semisimple part.
    If $S$ satisfies the conclusion of \cref{thm:mainGsemisimpleLp}, then $H^i(L,V)=0$ for every $0\le i\leq \rank(S)$. 
\end{proposition*}

We first show that the absence of almost invariant vectors for every non-compact almost simple factor is preserved under restriction to a Levi subgroup. We prove this for unitary representations (\cref{lem:levilegitunitary}) and use the Mazur map to extend it to $L^p$-isometric representations (\cref{lem:levilegit}). This allows us to use the assumption on the semisimple part of the Levi subgroup. The Shapiro isomorphism transfers the required cohomological properties to a cocompact lattice of the semisimple part. We then apply \cref{product of two groups} to this lattice and the center of the Levi subgroup and use \cref{cor:restriction_is_injective_cocolattice} to obtain the vanishing for the Levi subgroup.

\subsubsection{Absence of almost invariant vectors passes to Levi subgroups}

\begin{lemma} \label{lem:levilegitunitary}
    Let $G$ be a simply connected semisimple group and let $L<G$ be a Levi subgroup. Consider a unitary representation of $G$ on a Hilbert space $V$ whose restriction to every non-compact almost simple factor of $G$ has no almost invariant vectors.
    Then its restriction to every non-compact almost simple factor of $L$ has no almost invariant vectors.
\end{lemma}

\begin{proof}
Every non-compact almost simple factor of $L$ lies in an almost simple normal factor of~$G$. Restricting to this factor and its corresponding Levi subgroup, we may therefore assume that $G$ is almost simple over a single local field.
We also assume that $\rank(G)\geq 2$, otherwise either $L$ has no non-compact simple factors or $L=G$, and in any case the lemma follows trivially.
We let $S$ be a non-compact simple factor of $L$.
We argue to show that $S$ has no almost invariant vectors.
Being a non-compact simple factor of a Levi subgroup, $S$ contains a rank-one subgroup $H_\alpha$ associated with a root $\alpha$, as constructed in \cite[\S3]{oh:uniform}.
It is shown in \cite[Theorems 4.1 and 4.2]{oh:uniform} that for any unitary $G$-representation with no non-trivial $G^+$-invariants, the restriction to $H_\alpha$ is strongly $L^{4+\epsilon}$ for every $\epsilon>0$.
Note that $G=G^+$ because $G$ is simply connected.  

In particular, any unitary $G$-representation with no invariants has no $H_\alpha$-almost invariant vectors~\cite[Theorem~1.3]{Gorfine:spectral}.
It follows that the restriction to $S$ of any unitary $G$-representation with no invariants has no almost invariant vectors as well.
\end{proof}

\begin{lemma} \label{lem:levilegit}
    Let $G$ be a simply connected semisimple group acting by linear isometries on a separable $L^p$-space such that the restriction to every non-compact almost simple factor of $G$ has no almost invariant vectors. If $L<G$ is a Levi subgroup, then the restriction to every non-compact almost simple factor of $L$ has no almost invariant vectors.
\end{lemma}

\begin{proof}
For $p=2$ the lemma follows from \cref{lem:levilegitunitary}, so we assume $p\neq 2$.
By \cref{thm:BLM} the $G$-representation is given by $\Phi_p\circ \pi$ for  some continuous homomorphism $\pi\colon G\to \BL(X)$.
By \cref{lem:noaiBL}, the unitary representation $\Phi_2\circ \pi$ satisfies the same assumption on the non-compact almost simple factors of $G$. The case $p=2$ gives the corresponding conclusion for $L$, and another application of \cref{lem:noaiBL} transfers it back to $\Phi_p\circ \pi|_L$.
\end{proof}

\subsubsection{Conclusion of the proof}

\begin{proof}[Proof of \cref{prop:mainLconditions}]
Let $S$ be the semisimple part of $L$ and put $r=\rank(S)$. By assumption, $S$ satisfies the conclusion of \cref{thm:mainGsemisimpleLp}. See the remark on simple connectivity before \cref{thm:mainGsemisimpleLp}. 
Let $V$ be an $L^p$-space with a $G$-action as in the proposition.
By \cref{lem:Lpcofinal}, every continuous $L$-cochain takes values in a separable $G$-invariant $L^p$-subspace of $V$. Since the assumptions on the representation pass to such subspaces, we may assume that $V$ is separable.

By \cref{lem:levilegit}, the restriction to every non-compact almost simple factor of $S$ has no almost invariant vectors.
Since $L$ is a \emph{proper} Levi subgroup, $Z(L)$ is non-compact. By \cite[Theorem 6.1]{badergelander:equicontinuous}, $V^{Z(L)}=0$. 
The product $S\cdot Z(L)$, where $Z(L)$ is the center of $L$, is a subgroup of finite index in~$L$. 
Let $\Gamma<S$ be a cocompact lattice. By cocompactness, $\Gamma$ is of type $F_\infty$.
By \cref{lem:cocolatticeVsplits}, the ($S$-induced) representation $W=I_\Gamma^p(V)$ identifies with $L^p(S/\Gamma,V)$ with the diagonal $S$-action. By \cref{lem:spectral gap induction}, the restriction to every non-compact almost simple factor of $S$ still has no almost invariant vectors. By the assumption, $H^i(S,W)=0$ for $i<r$ and $H^r(S,W)$ is Hausdorff. The topological Shapiro isomorphism $H^i(\Gamma,V)\cong H^i(S,W)$ from \cref{thm:cocoshapiro} gives the same conclusions for $H^i(\Gamma,V)$.

By \cref{product of two groups}, we get that $H^i(\Gamma\cdot Z(L),V)=0$ for all $ i\leq r$.
Since $\Gamma\cdot Z(L)$ is a closed cocompact subgroup of $L$ such that $\sfrac{L}{\Gamma\cdot Z(L)}$ carries a non-zero $L$-invariant Radon measure, \cref{cor:restriction_is_injective_cocolattice} implies that $H^i(L,V)$ injects in $H^i(\Gamma\cdot Z(L),V)$, and so it vanishes for all $ i\leq r$.
This concludes the proof.
\end{proof}

\subsection{Cohomological applications}

From the simple and simply connected case, we can deduce a version of \cref{thm:mainGsimpleLp} for connected simple Lie groups (\cref{cor:vanishin_simple_connected}). We then deduce a similar statement for lattices (\cref{cor:main_simple_lattice}). We conclude the section with establishing Gromov's conjecture \cref{cor:Gromov}.

\begin{corollary}\label{cor:vanishin_simple_connected}
    Let $G$ be a connected simple Lie group with finite center and $r=\rank(G)\ge 2$. Let $V$ be an $L^p$-space on which $G$ acts by linear isometries without non-trivial invariant vectors. We have $H^i(G,V)=0$ for $0\leq i<r$, and $H^r(G,V)$ is Hausdorff.
\end{corollary}

\begin{proof}
    By \cref{lem:separable reduction}, we may and will assume that $V$ is separable. Let $Z=Z(G)$ be the finite center of~$G$ and $W=V^Z$. Averaging over $Z$ gives a $G$-equivariant contractive linear projection onto $W$, so $W$ is an $L^p$-space by \cref{lem:Lp_equivariant_projection}.

    As recalled in \cref{sec:semisimple_groups}, $G/Z$ is the identity component $\mathbf H(\mathbb R)^\circ$ for an adjoint simple real algebraic group $\mathbf H$. Let $\widetilde{\mathbf H}\to\mathbf H$ be its algebraically simply connected central cover and put $\widetilde G=\widetilde{\mathbf H}(\mathbb R)$. Then $\widetilde G$ is connected, and the projection $\widetilde G\to G/Z$ is surjective with finite central kernel. 
    We pull back $W$ to $\widetilde G$-representation. Since the kernels are finite, \cref{thm:partial_hochschild_serre} yields topological isomorphisms
    \[
    H^*(G,V)\cong H^*(G/Z,W)\cong H^*(\widetilde G,W).
    \]
     The $\widetilde G$-action on $W$ is isometric, $W^{\widetilde G}=V^G=0$, and $\rank(\widetilde G)=r$. Hence \cref{thm:mainGsimpleLp} implies both the desired vanishing and the Hausdorff property. 
\end{proof}

\begin{corollary} \label{cor:vanishin_simple_connected for semisimple}
    Let $G$ be a connected semisimple Lie group with finite center and $r=\rank G\ge 2$. Assume $G$ has property (T). Let $V$ be an $L^p$-space on which $G$ acts by linear isometries such that no non-compact factor of $G$ has non-trivial invariant vectors. Then we have $H^i(G,V)=0$ for $0\leq i<r$, and $H^r(G,V)$ is Hausdorff.
\end{corollary}

\begin{proof}
The reduction to separable $V$ and to the case where $G$ is the real points of simply connected algebraic group work exactly as in the previous proof. Then we apply \cref{thm:mainGsemisimpleLp} and use the fact that no non-compact factor has almost invariant vectors because each non-compact factor has property $(T)$ and has no invariant vectors. 
\end{proof}

The following contains \cref{thm:LpHi}.

\begin{corollary} \label{cor:main_simple_lattice}
    Let $G$ be a simply connected and simple group of characteristic 0 or a simple connected Lie group. Assume that $G$ has a finite center and $\rank(G)\geq 2$. Let $\Gamma<G$ be a lattice. If $V$ is an $L^p$-space on which $\Gamma$ is acting by linear isometries such that $V^\Gamma=0$, then $H^i(\Gamma,V)=0$ for every $i<\rank(G)$, and $H^{\rank(G)}(\Gamma,V)$ is Hausdorff.
\end{corollary}

\begin{proof}
    Let $V$ be an $L^p$-space with a $\Gamma$-action by linear isometries such that $V^\Gamma=0$. 
    Let $I^p(V)$ be the induced representation to $G$.
    By \cref{thm:shapiro},  $I_{\loc, \Gamma}^p(V)^G=0$ and so $I^p_\Gamma (V)^G=0$.

    By \cref{thm:mainGsimpleLp} in the simply connected case and \cref{cor:vanishin_simple_connected} in the connected Lie group case, $H^k\bigl(G,I^p_{\Gamma}(V)\bigr)=0$ for $0\leq k<\rank~G$.
    If $G$ is non-Archimedean, the lattice $\Gamma$ is cocompact, and by \cref{thm:cocoshapiro}, $H^k\bigl(G,I^p_{\Gamma}(V)\bigr)\cong H^k(\Gamma,V)$. The statement follows.

    Otherwise, by \cref{cor:shapiro_lemma_for_lattices}, $H^k(G,I^p_{\Gamma}V)$ surjects onto $H^k(\Gamma,V)$ for $k<\rank~G$, so the latter vanishes as well in that range. 
    The Hausdorffness in the rank is deduced from \cref{lem:uptrick_34} as follows: The lattice $\Gamma$ has $FP_\infty(\mathbb{R})$ by \cite{borelserre:corners} and the fact that $\Gamma$ is arithmetic, by \cite[Theorem~IX.1.11]{margulis:discrete}, and the class of $L^p$-spaces on which $\Gamma$ acts by linear isometries without invariants is closed under ultrapowers by \cref{lem:no_almost_invariants_ultrapowers}, since $\Gamma$ has property (T).
\end{proof}

We conclude with a proof of Gromov's conjecture \cref{cor:Gromov}.

\begin{corollary}[Gromov's conjecture] \label{cor:Lp_cohomology_alg}
If $G$ is a connected semisimple group with finite center of rank $r$, then  $H^i(G,L^p(G))=0$ for every $1\leq p<\infty$ and for every $1\leq i<r$. Furthermore, $H^r(G,L^p(G))$ is Hausdorff.
\end{corollary}

\begin{proof}
By \cref{thm:mainGsemisimpleLp}, we only need to show that there are no almost invariant vectors for any non-compact factor of $G$. 

Let $H$ be a non-compact factor of $G$. By \cref{lem:noaiBL} it is enough to consider the case $p=2$. As a $H$-representation we have that $L^2(G)=L^2(G/H)\bar\otimes L^2(H)$, where the tensor is the Hilbertian tensor product and $H$ acts trivially on $G/H$. Since $H$ is not amenable, there are no $H$-almost invariant vectors in $L^2(H)$ and so there are none in $L^2(G)$.
\end{proof}

\section{Actions on simplicial complexes} \label{sec:actions}

The first half of this section, \cref{subsec:6setup} and \cref{subse:6.2},
is devoted to the proof of \cref{thm:2ndversion}.
In its second half, \cref{subse:applic}, we will discuss various applications of this Theorem, coupled with vanishing of cohomology results that were proved in earlier sections.

\subsection{Setup} \label{subsec:6setup}

Let $X$ be a simplicial complex. 
Let $\Delta_\ast(X)$ denote the simplicial chain complex of $X$ with real coefficients. That is, $\Delta_k(X)$ is the real vector space spanned by the oriented $k$-simplices of $X$ subject to the relation $\sigma+\overline{\sigma}=0$ where $\sigma$ is any oriented $k$-simplex and $\overline{\sigma}$ is the same simplex with the opposite orientation. 
A choice of an orientation for each $k$-simplex gives a basis for $\Delta_k(X)$.
We consider the $\ell^1$-norm defined by this basis, which we denote by $\|\cdot\|$. 
This norm does not depend on the choice of the orientations.
From now on, we equip $\Delta_k(X)$ with this norm. 
We denote $\Delta_{-1}(X)=\mathbb{R}$, which we think of as the free vector space over the unique $(-1)$-simplex of $X$. The map $\Delta_0(X)\to \Delta_{-1}(X)$ is called the \emph{augmentation map} and is often denoted by $\epsilon$; we prefer to denote it by $\partial$, for homogeneity. We call the chain complex $\Delta_\ast(X)$ starting in degree $-1$ the \emph{augmented} simplicial chain complex of $X$. 

Let $\overline{\Delta}_k$ be the completion of $\Delta_k$ with respect to the $\ell^1$-norm.

\begin{lemma}\label{continuity of boundary map}
    The boundary maps of the augmented simplicial chain complex $\Delta_\ast(X)$ are bounded. In particular, their completions turn $\bar\Delta_\ast(X)$ into a chain complex of $\ell^1$-spaces. 
\end{lemma}
\begin{proof}
    If $\sigma$ is an oriented $k$-simplex then $\partial\sigma$ is a sum of $k+1$ simplices thus $\|\partial\sigma\|=k+1$. If $c=\sum_{i=1}^na_i\sigma_i\in\Delta_k(X)$ is chain, where each $\sigma_i$ is an oriented $k$-simplex with coefficient $a_i\in\mathbb{R}$, then $\|\partial c\|\leq\sum_{i=1}^n|\alpha_i|\|\partial \sigma_i\|=(k+1)\|v$.
\end{proof}

Every (simplicial) automorphism of $X$ acts as an isometry on $\Delta_k(X)$ and thus extends to an isometry of $\overline{\Delta}_k$. The boundary maps are equivariant with respect to the action of $\Aut(X)$. 

We say that a simplicial action of a topological group $G$ on $X$ is \textit{continuous} if the stabilizer of each vertex -- equivalently, each simplex -- is open in~$G$. In this case, the chain groups $\Delta_k(X)$ are \emph{discrete $G$-modules}, that is, they carry a linear $G$-action with open stabilizers. 
The following fact is immediate.

\begin{lemma} 
    If $G$ acts continuously and simplicially on a simplicial complex $X$, then the  induced action on each completed chain group $\overline{\Delta}_k(X)$ is a continuous,  linear and isometric. 
\end{lemma}

The goal of this section is to prove the following
generalization of \cref{thm:2ndversion}.

\begin{theorem}\label{thm:finite_orbit_main}
    Let $G$ be a topological group and $X$ an $\mathbb{R}$-acyclic and finite-dimensional simplicial complex equipped with a continuous and simplical action of~$G$.  
    If $H^k\bigl(G,\overline\Delta_k(X))=0$ for all $1\leq k\leq \dim X$, then there exists a finite $G$-orbit in $X$.
\end{theorem}

\subsection{Proof of \cref{thm:finite_orbit_main}} \label{subse:6.2}

The theorem will follow immediately from \cref{finite orbit}, which shows that under the same assumptions there exists a $G$-invariant vector in $\overline{\Delta}_0(X)$, and \cref{non existence of invariants}, which shows that this implies the existence of a finite orbit.


\begin{theorem} \label{finite orbit}
Let $X$ be an $n$-dimensional $\mathbb{R}$-acyclic simplicial complex and let $G$ be a topological group acting on $X$ simplicially and continuously. There exists $0\leq k\leq n$ such that  $H^k\bigl(G,\overline\Delta_k(X)\bigr)\neq 0$.
\end{theorem}

\begin{figure}[hbt!]
\centering
\begin{subfigure}[b]{0.51\textwidth}
\centering
\begin{tikzcd}[column sep=1.1em, row sep=1.2em]
	{C^n_n} & {C^{n-1}_n} & {\cdots} & {C^0_n} \\
	{C^n_{n-1}} & {C^{n-1}_{n-1}} & \cdots & {C^0_{n-1}} \\
	\vdots & \vdots && \vdots \\
	{C^n_0} & {C^{n-1}_0} & \cdots & {C^0_0} \\
	{C^n_{-1}} & {C^{n-1}_{-1}} & \cdots & {\mathbb{R}}
	\arrow["\partial"', from=1-1, to=2-1]
	\arrow["\delta"', from=1-2, to=1-1]
	\arrow["\partial", from=1-2, to=2-2]
	\arrow["\delta"', from=1-3, to=1-2]
	\arrow["\delta"', from=1-4, to=1-3]
	\arrow["\partial", from=1-4, to=2-4]
	\arrow[from=2-1, to=3-1]
	\arrow["\delta"', from=2-2, to=2-1]
	\arrow[from=2-2, to=3-2]
	\arrow["\delta"', from=2-3, to=2-2]
	\arrow["\delta"', from=2-4, to=2-3]
	\arrow[from=2-4, to=3-4]
	\arrow[from=3-1, to=4-1]
	\arrow[from=3-2, to=4-2]
	\arrow[from=3-4, to=4-4]
	\arrow["{\partial }", from=4-1, to=5-1]
	\arrow[from=4-2, to=4-1]
	\arrow["\partial", from=4-2, to=5-2]
	\arrow[from=4-3, to=4-2]
	\arrow[from=4-4, to=4-3]
	\arrow["\partial", from=4-4, to=5-4]
	\arrow[from=5-2, to=5-1]
	\arrow[from=5-3, to=5-2]
	\arrow["0"', from=5-4, to=5-3]
\end{tikzcd}
\caption{The double complex.}
\label{fig: total complex diagram}
\end{subfigure}%
\hfill
\begin{subfigure}[b]{0.45\textwidth}
    \centering
    \begin{tikzcd}[column sep=0.8em, row sep=1.15em]
	0 & {c_n} \\
	& {\partial c_n} & {c_{n-1}} \\
	{} &&& {c_2} \\
	&&& {\partial c_2} & {c_1} \\
	&&& 0 & {\partial c_1} & {c_0} \\
	&&&& 0 & 1
	\arrow[from=1-2, to=1-1]
	\arrow["\partial"', from=1-2, to=2-2]
	\arrow["\delta"', from=2-3, to=2-2]
	\arrow["\cdots"{marking, allow upside down}, draw=none, from=3-4, to=2-3]
	\arrow["\partial"', from=3-4, to=4-4]
	\arrow[dashed, from=4-4, to=5-4]
	\arrow["\delta"', from=4-5, to=4-4]
	\arrow["\partial"', from=4-5, to=5-5]
	\arrow[dashed, from=5-5, to=5-4]
	\arrow[dashed, from=5-5, to=6-5]
	\arrow["\delta"', from=5-6, to=5-5]
	\arrow["\epsilon"', from=5-6, to=6-6]
	\arrow[dashed, from=6-6, to=6-5]
\end{tikzcd}

    \caption{The choice of the $c_k$.}
    \label{fig:def_of_c_k}
\end{subfigure}
\caption{On the left, $C_j^k=C^k\bigl(G,\bar\Delta_j(X)\bigr)$, with $\bar\Delta(X)_{-1}=\mathbb{R}$; $\partial$ is induced by the simplicial boundary and $\delta$ is the coboundary map. The right hand side illustrates how the cochain $c_k$ arise.  }
\label{fig:cochains-and-lifts}
\end{figure}
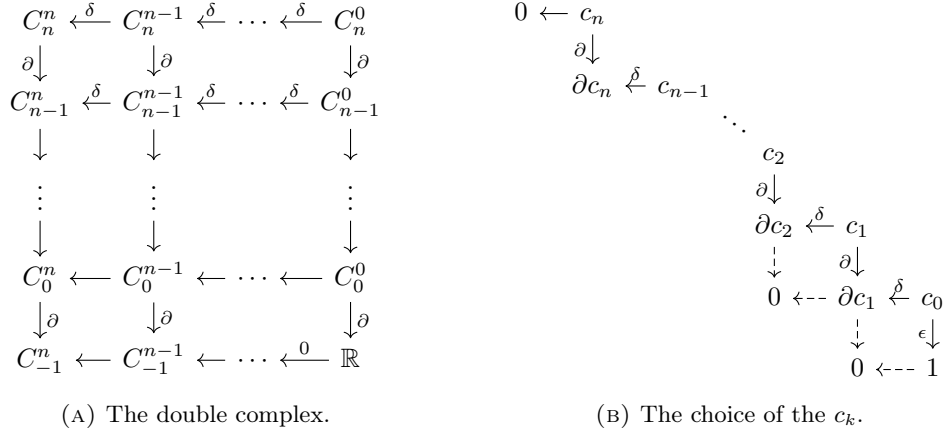




\begin{proof}
We assume that $H^k\bigl(G,\overline\Delta_k(X)\bigr)=0$ for $1\leq k\leq n$ and we will show that $H^0\bigl(G,\overline\Delta_0(X)\bigr)\neq0$. As in \cref{sec:contcohom} we denote the differentials in $C^\ast\bigl(G,\Delta_k(X)\bigr)$ by~$\delta$. The differentials $\partial$ in $\Delta_\ast(X)$ turn $C^k\bigl(G,\Delta_\ast(X)\bigr)$ into a chain complex whose differentials we also denote by~$\partial$. So $\delta$ raises and $\partial$ lowers the degree. 
We have $\partial\delta=\delta\partial$ and $\delta(1)=0$, where $1$ means the constant $0$-cocycle $G\to \Delta_{-1}(X)=\bbR$.

Since $X$ is $\mathbb{R}$-acyclic, the augmented chain complex $\Delta_\ast(X)$ is exact. For each $k$, the subcomplex of locally constant cochains in $C^k\bigl(G,\Delta_\ast(X)\bigr)$, which can be identified with locally constant maps $G^k\to \Delta_\ast(X)$, is also exact: A locally constant map into the image of $\partial$ lifts via $\partial$ to a locally constant map by choosing a preimage for each value.
By using this exactness and an increasing induction (cf.~\cref{fig:def_of_c_k}), we construct locally constant cochains $c_k\in C^k\bigl(G,\Delta_k(X)\bigr)$ such that 
\begin{equation*}
\partial c_0=1,~~
\partial c_k=\delta c_{k-1}.
\end{equation*}
for $k\ge 1$. We find $c_0$ with $\partial c_0=1$ by exactness. Given $c_{k-1}$ for $k\ge 1$, we have $\partial\delta c_{k-1}=\delta\partial c_{k-1}=0$, using $\delta(1)=0$ for $k=1$ and $\delta^2=0$ for $k>1$. So $c_k$ with $\partial c_k=\delta c_{k-1}$ exists by exactness.
Because of $\dim X=n$, we have $c_k=0$ for $k\ge n+1$ and thus $\delta c_n=0$.

Next we regard the $c_k$ as cochains with values in $\overline\Delta_k(X)$. Set $a_{n+1}=0$. By descending induction, the assumption $H^k\bigl(G,\overline\Delta_k(X)\bigr)=0$ allows us to construct cochains $a_k\in C^{k-1}\bigl(G,\overline\Delta_k(X)\bigr)$ such that 
\[
\delta a_k=c_k-\partial a_{k+1}~~\text{for $k=n,\dots,1$}
\]
because in each inductive step the right-hand side is a cocycle: for $k=n$ this follows from $\delta c_n=0$, and for $k<n$ we have
\[
\delta(c_k-\partial a_{k+1})
=\delta c_k-\partial\delta a_{k+1}=\partial c_{k+1}-\partial(c_{k+1}-\partial a_{k+2})=0.
\]
The same computation for $v=c_0-\partial a_1$ shows $\delta v=0$. As $\ker\delta=H^0(G,\bar{\Delta}_0(X))$ and $v$ is non-zero because of $\partial v=1$, the result follows.  
\end{proof}

\begin{proposition}\label{non existence of invariants}
    Assume $G$ acts simplicially on a simplicial complex $X$. If for some $k\geq0$ there is a non-zero $G$-invariant vector in $\overline{\Delta}_k(X)$, then $G$ has a finite orbit.
\end{proposition}
\begin{proof}

    Let $k\geq0$. Choosing an orientation for each simplex gives an isometric isomorphism between $\overline{\Delta}_k$ and the $\ell^1$-space over the set of $k$-simplices. 
    Let $v=\sum a_i\sigma_i\neq 0$ be a $G$-invariant vector in $\overline{\Delta}_k$ under this identification. There exists some $M>0$ for which $\{\sigma_i\mid \abs{a_i}\geq M\}$ is not empty. This set is $G$-invariant and finite, since the $\ell^1$-norm of $v$ is finite. Thus, the action of $G$ on the set of $k$-simplices has a finite orbit, so the action of $G$ on $X$ has a finite orbit as well.
\end{proof}

\begin{proof}[Proof of \cref{thm:finite_orbit_main}]
    \cref{finite orbit} gives us a non-trivial $G$-invariant vector in $\overline\Delta_0$ and \cref{non existence of invariants} applied for $k=0$ gives us a finite orbit.
\end{proof}

\subsection{Geometric applications} \label{subse:applic}

We are finally in a position to discuss
applications of the various vanishing of cohomology results that were proved in earlier sections, mostly \cref{sec: vanishing}.

Using \cref{cor:main_simple_lattice} and \cref{thm:finite_orbit_main} we obtain (a generalization of) Farb's conjecture, \cref{mcor:farb}, for simple groups.

\begin{theorem} \label{thm:farb_simple}
    Let $G$ be a simple group of characteristic 0 and let $\Gamma<G$ be a lattice. Suppose $X$ is an $\mathbb{R}$-acyclic simplicial complex satisfying $\dim X<\rank(G)$.
    Any simplicial $\Gamma$-action on $X$ has a finite orbit.
    In particular, if $X$ is in addition a CAT(0) space, then there exists a $\Gamma$-fixed point.
\end{theorem}

\begin{proof}
    By \cite[Proposition~2.10]{platonovrapinchuk:algebraic}, there exists a map $\pi\colon \tilde{G}\to G$ with finite kernel, where $\tilde G$ is a simply connected simple group of characteristic 0. By the exact sequence \cite[I.5.5~Proposition 38]{serre:cohomologie} and the finiteness of $H^1(k,\ker\pi)$ \cite[Th\'eor\`eme~6.1]{BorelSerre:theoreme}, the image has finite index in $G$.
    Thus, it is enough to prove the statement under the assumption that $G$ is simply connected.
    
    Consider the action of $\Gamma$ on the $L^1$-space $V=\bigoplus_{k=0}^{\dim X}\overline \Delta_k(X)$. 
    If $V^\Gamma\neq 0$, then, in particular,  $(\overline\Delta_k)^\Gamma\neq 0$ for some $k$. By \cref{non existence of invariants}, $\Gamma$ has a finite orbit in $X$.
    Otherwise, by \cref{cor:main_simple_lattice}, $H^k(\Gamma,V)=0$ for all $0\leq k\leq \dim X$, and since cohomology commutes with direct sums, $H^k(\Gamma,\overline{\Delta}_k)=0$. By \cref{thm:finite_orbit_main}, $\Gamma$ has a finite orbit in $X$.
\end{proof}

We now work toward a semisimple version of the above, which will imply \cref{mcor:farb}. It will apply to irreducible lattices in connected semisimple Lie groups with property (T). We will need to consider the case in which the induced action of the ambient group on $I^1\overline{\Delta}_k$ satisfies that there is a non-compact factor that has invariant vectors. The other case will follow from \cref{thm:mainGsemisimpleLp}, as in \cref{thm:farb_simple}.

\cref{prop:continuous_vector_invariant} relates invariant vectors and \emph{continuous vectors}, it essentially shows that if a factor has invariant vectors, then the other factors have continuous vectors. So we will study that.

\begin{definition}\label{def:continuous_vectors}
    Let $G$ be an lcsc group and $\iota:\Gamma\to G$ with a dense image. Let $V$ be an $L^p$-space on which $\Gamma$ acts linearly and isometrically.     
    We say $v\in V$ is \emph{$G$-continuous} if $\lim_{n\to\infty}||\gamma_nv-v||_p=0$, whenever $\iota(\gamma_n)\to 1$ in $G$.

    The space of $G$-continuous vectors is a closed and $\Gamma$-invariant subspace.
\end{definition}

The following is \cite[Proposition~5.1]{baderboutonnethoudayerpeterson:charmenability}, where it was stated and proved for Hilbert spaces, but the proof works as is for $L^p$-spaces using the $L^p$-induction defined in \cref{sec:shapiro}.

\begin{proposition}    \label{prop:continuous_vector_invariant}
Let $G$ be an lcsc group such that \[1\to G_2\to G\xrightarrow[]{\iota}G_1\to  1\]
    is an exact sequence. Let $\Gamma<G$ be a lattice such that $\iota(\Gamma)$ is dense in $G_1$.
    Let $V$ be an $L^p$-space on which $\Gamma$ acts linearly and isometrically. Let $V_\iota$ denote the set of $G_1$-continuous vectors with respect to $\iota(\Gamma)<G_1$.
    
    There is a $G$-equivariant surjective isometry $\kappa:V_\iota\to I^p(V)^{G_2}$.
\end{proposition}

\begin{theorem} \label{thm:invariant_vectors_implies_finite}
    Let $G$ be a connected semisimple Lie group and let $\Gamma<G$ be an irreducible lattice. Suppose $X$ is a non-empty set with a $\Gamma$-action. Let $1\leq p<\infty$ and consider the induced action of $\Gamma$ on $\ell^p(X)$. If there exists a non-compact factor $H<G$ that has invariant vectors in the induced representation $I^p\bigl(\ell^p(X)\bigr)$, then $\Gamma$ has a finite orbit.
\end{theorem}

\begin{proof}
Let $\iota\colon G\to G/H$.
By \cref{prop:continuous_vector_invariant}, the closed and $\Gamma$-invariant subspace $V\subset \ell^p(X)$, of $G/H$-continuous vectors, is non-zero. Let $f\in V\bs\{0\}\subset \ell^p(X)$ be a non-zero vector. From the $\ell^p$-condition we get that $0$ is the only possible accumulation points of $f(X)$ outside $f(X)$. 
Let $x_0\in X$ be such that $f(x_0)\ne 0$. 
The map 
\[ g\colon \Gamma\to \bbR,~~\gamma\to (\gamma f)(x_0)=f(\gamma^{-1}x_0)\]
extends to a continuous map $G/H\to \bbR$. Since $\Gamma$ is dense in $G/H$, this  extension takes values in the closure of $f(X)$, which is a countable set. Since $G/H$ is connected, the extension is constant. This implies that the $\Gamma$-orbit of~$x_0$ is finite by the $\ell^p$-condition. 
\end{proof}

We can now deduce \cref{thm:fixedsemisimple} from the introduction.

\begin{corollary} \label{cor:farb_for_semisimple}
    Let $G$ be a connected semisimple Lie group with finite center and no compact factors. Assume $G$ has property (T). Let $\Gamma<G$ be an irreducible lattice. Suppose $X$ is an $\mathbb{R}$-acyclic simplicial complex and $\Gamma$ acts on $X$ simplicially. 
    If $\dim X<\rank(G)$, then $\Gamma$ has a finite orbit.
\end{corollary}

\begin{proof}
We deduce from \cref{cor:vanishin_simple_connected for semisimple} that for any $L^p$-space $V$ and a linear isometric action on $V$  without invariant for non-compact factors, we have $H^i(G,V)=0$ for $0\leq  i< \rank(G)$ and $H^{\rank(G)}(G,V)$ is Hausdorff.

    Consider the action of $\Gamma$ on the $L^1$-space $V=\bigoplus_{k=0}^{\dim X}\bar \Delta_k(X)$. 
    Consider the $G$-action on $I^1(V)$. If some non-compact factor of $G$ has invariant vectors, then by \cref{thm:invariant_vectors_implies_finite}, $\Gamma$ has a finite orbit in one of the $\bar\Delta_k(X)$, and so in $X$.

    Assume otherwise. As in \cref{cor:vanishin_simple_connected} since all factors have property (T), the action of $G$ on $I^1V$, satisfies that there are no almost invariant for non-compact factors, and so $H^k\bigl(G,I^1(V)\bigr)=0$ for every $0<k<\rank(G)$. By \cref{cor:shapiro_lemma_for_lattices}, $H^k(\Gamma,V)=0$ for every $0<k<\rank(G)$. In particular, $H^k\bigl(\Gamma,\bar{\Delta}_k(X)\bigr)=0$ for every $0<k<\rank(G)$. By \cref{thm:finite_orbit_main}, $\Gamma$ has a finite orbit in $X$.
\end{proof}

The following is a reproduction of \cref{thm:padicquestion}.

\begin{corollary} \label{thm:padicquestionagain}
    For a simple $p$-adic group $G$ of rank $r$ and a lattice $\Gamma<G$, $d=r$ is the minimal number for which $\Gamma$ has an action on a contractible $d$-dimensional simplicial complex where all orbits are infinite.
\end{corollary}

\begin{proof}
    The Bruhat-Tits building of $G$ is of dimension $r$ and $G$ acts on it properly. Thus, so does $\Gamma$, hence all orbits of $\Gamma$ are infinite.  
    By \cref{cor:main_simple_lattice} and \cref{thm:farb_simple}, there are no such actions on $\mathbb{R}$-acyclic (and, in particular, contractible) simplicial complexes of dimension smaller than the rank.
\end{proof}

Finally, we provide the following two applications which do not use \cref{sec: vanishing} at all.




\begin{corollary}
Let $G$ be a discrete group of geometric dimension $n$ and let $X$ be a universal cover of a simplicial $K(G,1)$ of dimension $n$.
There exists $0\leq k\leq n$ such that $H^k(G,\overline\Delta_k)\neq 0$.

Furthermore, if $G$ is of type $F$ there exists $0\leq k\leq n$ such that $H^k\bigl(G,\ell^1(G)\bigr)\neq 0$.
\end{corollary}

\begin{proof}
    Since $G$ acts simplicially on $X$, $X$ is $\mathbb{R}$-acyclic, and all orbits are infinite, the result follows from \cref{thm:finite_orbit_main}.
    
    Assume $G$ is of type $F$. Since the action of $G$ on $X$ is free, with finitely many orbits, $H^k(G,\overline{\Delta}_k)=\bigoplus_{j=1}^{m_k}H^k(G,\ell^1(G))$. Since for some $k$ the cohomology does not vanish, we get that  $H^k(G,\ell^1(G))\neq 0$ as well.
\end{proof}

The analogous statement for $\ell^2$-cohomology is called the \emph{zero-in-the-spectrum question} and it is open; see \cite{lott:zero} for a review on the subject.

\begin{theorem}
    If $\Gamma_1,\dots,\Gamma_n$ are discrete groups of type $FP_{n}(\mathbb{R})$ with property (T), then any simplicial action of $\Gamma=\Gamma_1\times \cdots \times \Gamma_n$ on a simplicial $\mathbb{R}$-acyclic complex $X$ with $\dim X\leq n$ satisfies that there exists $1\leq i\leq n$ such that $\Gamma_i$ has a finite orbit.
\end{theorem}

\begin{proof}
    Consider the action of $\Gamma$ on the $\ell^1$-space $V=\bigoplus_{k=0}^{\dim X}\overline \Delta_k(X)$. 
    If for some $1\leq i\leq n$,  $\Gamma_i$ has almost invariant vectors in $V$, then by property (T) and \cite[Theorem~A]{baderfurmangelandermonod:BFGM:property}, $\Gamma_i$ has non-trivial invariant vectors in $V$. In particular, in  $\overline\Delta_k(X)$ for some $k$. By \cref{non existence of invariants}, $\Gamma_i$ has a finite orbit in $X$.

    Otherwise, by \cref{cor:products_of_discrete}, $H^k(\Gamma,V)=0$ for all $0\leq k\leq n$, and since cohomology commutes with direct sums, $H^k(\bigl(\Gamma,\overline{\Delta}_k(X)\bigr)=0$. By \cref{thm:finite_orbit_main}, $\Gamma$ has a finite orbit in $X$, and in particular, every $\Gamma_i$ does.
\end{proof}


\bibliographystyle{amsalpha-arxiv}
\bibliography{biblio}

\end{document}